\documentclass[hidelinks,onefignum,onetabnum]{siamart250211}

\usepackage{colortbl}
\usepackage{mathtools}
\usepackage{tikz}
\usepackage{enumitem}
\usepackage{chemfig} \newcommand{\cf}[1]{\textsf{#1}}

\usetikzlibrary{positioning}
\usetikzlibrary{arrows}
\usetikzlibrary{decorations.pathmorphing}
    \tikzset{%
    fwdrxn/.style={very thick, arrows={-Stealth[length=5pt,width=5pt]}},
    revrxn/.style={very thick, arrows={-Stealth[length=5pt,width=5pt,left]}},
    newt/.style={turq, opacity=0.15}
    }

\usepackage{lipsum}
\usepackage{amsfonts}
\usepackage{amssymb}
\usepackage{graphicx}
\usepackage{epstopdf}
\usepackage{algorithmic}
\ifpdf
  \DeclareGraphicsExtensions{.eps,.pdf,.png,.jpg}
\else
  \DeclareGraphicsExtensions{.eps}
\fi

\definecolor{newred1}{HTML}{D83034} 
\definecolor{newblue1}{HTML}{0041C2}
\definecolor{reverse-purple}{HTML}{C701FF}

\newsiamremark{remark}{Remark}
\newsiamremark{example}{Example}%
\newsiamremark{hypothesis}{Hypothesis}
\crefname{hypothesis}{Hypothesis}{Hypotheses}
\newsiamthm{claim}{Claim}
\newsiamremark{fact}{Fact}
\crefname{fact}{Fact}{Facts}

\headers{Asymptotic Robustness}{H. Hong, D. Rojas La Luz, G. Craciun}

\title{Asymptotic Robustness in Biochemical Systems\thanks{
\funding{This work was supported in part by the Fulbright Program, which sponsored D.~RLL.'s doctoral studies in the United States. G.~C. and D.~RLL. have been also funded in part by the National Science Foundation under grant DMS-2051568.}}}

\author{Hyukpyo Hong\thanks{Department of Mathematics, University of Wisconsin--Madison, Madison, WI, USA (\email{hhong78@wisc.edu}, \email{rojaslaluz@wisc.edu}, \email{craciun@wisc.edu}).}
\and Diego Rojas La Luz\footnotemark[2]
\and Gheorghe Craciun\footnotemark[2] \thanks{Department of Biomolecular Chemistry, University of Wisconsin--Madison, Madison, WI, USA}}

\usepackage{amsopn}

\ifpdf
\hypersetup{
  pdftitle={Asymptotic Robustness in Biochemical Systems},
  pdfauthor={H. Hong, D. Rojas La Luz, G. Craciun}
}
\fi

\begin{document}

\maketitle

\begin{abstract}
Living systems maintain stable internal states despite environmental fluctuations. Absolute concentration robustness (ACR) is a striking homeostatic phenomenon in which the steady-state concentration of a species remains invariant despite changes in total supply. Although experimental studies have reported approximate---but not exact---robustness in steady-state concentrations, such behavior has often been attributed to exact ACR motifs perturbed by measurement noise or minor reactions, rather than recognized as a structural property of a network itself. Here, we introduce a previously underappreciated phenomenon, namely {\em asymptotic ACR} ({\em aACR}): approximate robustness can emerge solely from the network structure, without requiring exact ACR motifs or negligible parameters. We find that aACR is more pervasive than classical ACR, as demonstrated in systems such as the \textit{Escherichia coli} EnvZ-OmpR osmoregulation system and a futile cycle. Furthermore, we  prove that such ubiquity stems solely from network structure without fine-tuning of kinetic parameters. The notion of aACR provides a rigorous and practical tool to analyze robust responses in broad biochemical systems.
\end{abstract}



\section{Introduction}\label{sec:Introduction}
Biological systems possess a remarkable ability to maintain precise concentrations of key molecules despite environmental fluctuations, and it is essential for sustaining life~\cite{kitano2004biological,kitano2007towards, alon1999robustness, aoki2019universal, barkai1997robustness, Goldstein-homeostasis, alon2019introduction}. This robustness is observed across diverse biochemical contexts, including the Calvin cycle \cite{Wilson-homeostasis}, the EnvZ-OmpR osmoregulation system, various metabolic pathways in {\it Escherichia coli} (\textit{E. coli}) \cite{batchelor2003robustness, laporte1985compensatory, ishii2007multiple}, and pyrimidine metabolism in plants \cite{Leo2021Mechanisms, Michael2005Functional, Peter2005Inhibition}. 
A central challenge in systems biology is to understand how such robustness arises from the underlying network of biochemical systems. 
Various mathematical frameworks have uncovered structural mechanisms that confer such robustness \cite{HART2013213, wang2021structure, shinar2007input, shinar2009sensitivity}, including control-theoretic approaches for robust perfect adaptation \cite{KHAMMASH2021509, gupta2022universal, aoki2019universal, araujo2018topological} and structural sensitivity analyses that are independent of reaction kinetics \cite{PhysRevE.96.022322SM, PhysRevLett.117.048101SM, Hong-fluxRPA, MOCHIZUKI2015189, doi:10.1002/mma.3436, Ferjani2018SciRep}. 

A notable milestone in this line of work is the concept of {\em absolute concentration robustness} (ACR), introduced by Shinar and Feinberg~\cite{ShinarFeinberg2010Science, gunawardena2010biological}. Specifically, a species is said to have ACR if the steady-state concentration of the species remains exactly invariant to changes in initial conditions. 
This sparked a series of investigations on necessary conditions for ACR~\cite{SHINAR201139, eloundou2016network}, its dynamical stability~\cite{joshi-craciun-2022foundations, joshi-craciun-2023-dynamic-acr}, stochastic formulations~\cite{Anderson2017FiniteRobustness, anderson2014stochastic}, and broader theoretical frameworks~\cite{Karp2012complex, GARCIAPUENTE2025102398}. However, experimental observations often reveal a subtler reality: concentrations that are not \textit{strictly} but \textit{approximately} invariant across wide ranges of environmental and internal perturbations \cite{batchelor2003robustness, laporte1985compensatory, kim2012need, Khammash2021PerfectBiology}.

This raises a fundamental question: what underlies this approximate robustness? Is it merely an experimental artifact, or could it reflect a deeper, structural property of the system? These questions remained unresolved because existing theoretical frameworks focused almost exclusively on \textit{exact} robustness, leaving \textit{approximate} robustness relatively unexplored. 
To bridge this conceptual gap, we present a theoretical framework that explains how near-robustness can emerge solely from network structure. Specifically, we here introduce the concept of \textit{asymptotic} ACR (aACR) and mathematically prove that approximate robustness can arise purely from the architecture of a reaction network, without fine-tuning of kinetic parameters and even without a core ACR module. 

This article is structured as follows. 
Section~\ref{sec:biology} demonstrates that aACR more accurately reflects experimental observations than conventional ACR. 
Section~\ref{sec:background} introduces setting and notation. 
Section~\ref{sec:main} illustrates the concept of aACR using a simple, representative network. We also reveal that aACR is ubiquitous and thus substantially more common than conventional ACR, as evidenced by detailed analysis of the \textit{E. coli} EnvZ-OmpR osmoregulation.
Finally, Section~\ref{sec:math-analysis} contains our main theorems that implies this robustness phenomenon arises generically from structural features of biochemical networks, without requiring parameter fine-tuning. Along with the theorems, we provide a general approach based on computational algebra that allows us to identify whether specific input-output pairs of species lead to the the aACR property or not.

\section{Biological motivation: Realistic experimental outputs require consideration of approximate concentration robustness}\label{sec:biology}

\begin{figure}[H]
    \centering
    \includegraphics[width=0.8\linewidth]{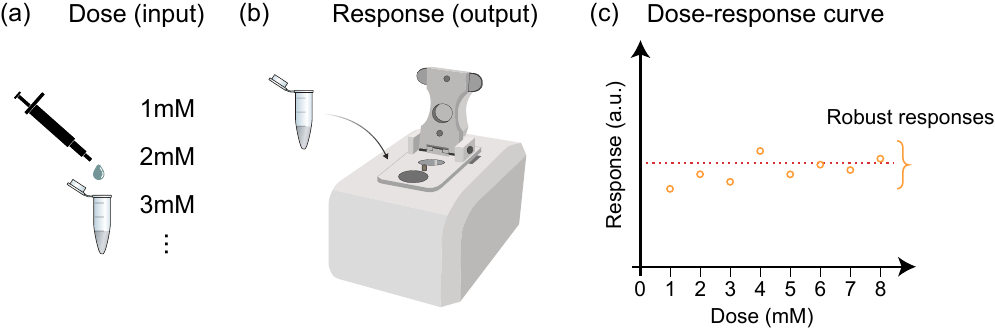}
    \caption{\textbf{Illustration of an experimental setup for assessing concentration robustness.} 
    (\textbf{a}) A series of experiments is conducted by varying the concentration of the dose (i.e., input species). 
    (\textbf{b}) The corresponding steady-state concentrations of the response (i.e., the output species) are measured. 
    (\textbf{c}) The input–output pairs are used to construct a dose–response curve. If the output concentrations remain similar despite changes in the input, the system is considered to exhibit concentration robustness of the output species with respect to the input species.}
    \label{fig:dose-response}
\end{figure}    

A common experimental approach to studying concentration robustness involves measuring how the steady-state concentration of an output species \( X \) responds as the initial concentration of an input species \( Y \) is systematically varied. This procedure yields a dose-response curve, offering insights into the robustness of the system. Specifically, the initial concentration of \( Y \) is systematically varied (e.g., 1 mM, 2 mM, 3mM, $\ldots$) (Fig.~\ref{fig:dose-response}a), and the respective steady-state concentration of \( X \) is measured using a detection device (Fig.~\ref{fig:dose-response}b). The resulting data points are plotted to create a dose-response curve, which reveals how \( X \) responds to changes in \( Y \). If the measured concentrations of \( X \) are nearly identical across all test tubes, despite the varying initial concentrations of \( Y \), one can conclude that \( X \) exhibits concentration robustness, as its concentration remains stable under perturbations to \( Y \) (Fig.~\ref{fig:dose-response}c). 
This experimental setup motivates two key considerations for defining and analyzing concentration robustness. First, in real-world experiments, it is often impractical to vary all initial conditions simultaneously. Instead, experimental protocols typically involve changing one input species at a time while holding others constant \cite{haas2012vivo}. This practical consideration suggests that robustness should be studied with respect to {\em specific} input species or conserved quantities, rather than  requiring invariance across {\em all} possible initial conditions. Second, experimental noise and the limited resolution of measurements make it difficult to distinguish between exact robustness (where the output concentration is exactly the same) and approximate robustness (where the output concentration is nearly but not exactly the same). In both cases, the dose-response curve (i.e., input-output curve) may appear identical, as shown in Fig.~\ref{fig:dose-response}c.  

These observations highlight the need for a more nuanced definition of concentration robustness that captures both exact and approximate behaviors. While traditional definitions of ACR require the output concentration to be exactly the same across all initial conditions, this condition is often too stringent to apply in experimental settings. Instead, we propose a new framework that focuses on \textit{aACR}, where the output concentration asymptotically approaches a fixed value as the input concentration is increased. 

\section{Background}\label{sec:background}
Here, we introduce the basic setup and definitions involving reaction networks in the deterministic regime where a dynamical system is modeled by a system of ordinary differential equations (ODEs).

Consider a biochemical system with $d$ chemical \emph{species} $\{S_1, \ldots, S_d\}$ with which finite number of reactions occur. For the $i$th \emph{reaction}, the vectors representing the number of each chemical species consumed or produced in one instance of that reaction are denoted by $\nu_i \in \mathbb{Z}_{\geq 0}^d$ and $\nu_i'\in \mathbb{Z}_{\geq 0}^d$, respectively. The reaction is often denoted by $\nu_i \to \nu_i'$.
%
Note how from this perspective the reactions are edges between vertices embedded in $\mathbb{Z}_{\geq 0}^d$, where the dimension is given by the number of species.  
With these notions we define a chemical reaction network, which is a graph representation of a system of ODEs. The formal definition is as follows. 

\begin{definition}\label{def:E-graph}
A \emph{chemical reaction network} is a finite simple directed graph $G=(V, E)$, where $V \subseteq \mathbb{Z}_{\geq 0}^d$ is the set of vertices, and $E \subseteq V \times V$ is the set of directed edges. Given an edge $(\nu,\nu')\in E$ we often write $\nu \to \nu'\in E$.
\end{definition}

We introduce more definitions for chemical reaction networks.

\begin{definition}\label{def:stoichiometric}
    The \emph{stoichiometric subspace} for a chemical reaction network $(V,E)$ is defined as
    $$S:={\rm{span}}\{\nu'_i-\nu_i:\nu_i \to \nu'_i \in E\}.$$ 

    For $\mathbf{x}_0 \in \mathbb{R}^d_{>0}$, the \emph{stoichiometric compatibility class} containing $\mathbf{x}_0$ is $\mathbf{x}_0+S$. In other words, the stoichiometric compatibility class containing $\mathbf{x}_0$ is the maximal set that can be reached by the system of ODEs which started from the initial condition $\mathbf{x}_0$.
\end{definition}

In order to connect a chemical reaction network to a system of ODEs, we need to employ a \emph{kinetic law} that characterizes the rule of determining the system of ODEs based on the network structure. Here, we introduce the most fundamental kinetic law, called the \emph{mass-action kinetics}, for the deterministic regime.

\begin{definition}\label{def:MAK}
The deterministic mass-action system associated with a chemical reaction network $(V,E)$ with rate constants $\{\kappa_i\}_{i=1,\ldots,r}$ is the following system of ODEs:
\begin{equation}\label{eq:MAK-ode}
    \frac{d\mathbf{x}(t)}{dt} = \sum_{i=1}^r \kappa_i (\nu_i' - \nu_i)f_i(\mathbf{x}(t))
\end{equation}
with the choice of function $f_i$:
\begin{equation}\label{eq:MAK-f-choice}
    f_i(\mathbf{x}) = \mathbf{x}^{\nu_i} \coloneqq x_1^{\nu_{i1}}\times \cdots \times  x_d^{\nu_{id}}
\end{equation}
where $r = |E|$ and $\mathbf{x}(t)=(x_1(t), \ldots, x_d(t)) \in \mathbb{R}^d_{\geq 0}$ is the vector that represents the amount of the species $S_1, \ldots, S_d$ at time $t$.
\end{definition}

Here, we illustrate these definitions using the following example with two species and two reactions: 
\begin{equation}\label{crn:archetypal}
    \vcenter{\hbox{\begin{tikzpicture}
            \begin{scope}[shift={(0,0)}]
            \node (1a) at (0,0) [left]  {$\cf{X} + \cf{Y}$} ; 
            \node (1b) at (1.25,0) [right] {$2\,\cf{Y}$,}; 
            \draw[fwdrxn, transform canvas={yshift=0.00pt}] (1a)--(1b) node [midway, above] {$\alpha$};
            \end{scope}
            \begin{scope}[shift={(3,0)}]
        \node (2a) at (0,0) [left] {$\cf{Y}$}; 
        \node (2b) at (1.25,0) [right] {$\cf{X}$.};
        \draw[fwdrxn, transform canvas={yshift=0.00pt}] (2a)--(2b) node [midway, above] {$\beta$};
    \end{scope}
        \end{tikzpicture}}}
\end{equation}

Since the first reaction $X+Y \to 2Y$ consumes one instance of chemical species $X$ and $Y$ and then produces two instances of $Y$, then it can be represented as $\nu_1 \to \nu_1'$ where $\nu_1 = [1, 1]^\intercal$ and $\nu_1'=[0,2]^\intercal$.  Similarly, $Y\to X$ is represented by $\nu_2 \to \nu_2'$ where $\nu_2 = [0, 1]^\intercal$ and $\nu_2'=[1,0]^\intercal$. Thus, this reaction network is given by the graph $(V,E)$ embedded in $\mathbb{Z}_{\geq 0}^2$, where $V = \{\nu_1,\nu_1',\nu_2,\nu_2'\}$ and $E = \{\nu_1 \to \nu_1', \nu_2 \to \nu_2'\}$.
Furthermore, the stoichiometric subspace of the chemical reaction network is given by 
\begin{equation*}
    S = \operatorname{span}\{[-1,1]^\intercal, [1,-1]^\intercal\} = \{(s,-s) : s\in \mathbb{R}\} \subset \mathbb{R}^2.
\end{equation*}
In the diagram~\eqref{crn:archetypal}, two parameters $\alpha$ and $\beta$ represent rate constants of the two reactions, respectively, i.e., $\kappa_{X+Y\to2Y} = \alpha$ and $\kappa_{Y \to X} = \beta $. Therefore, the mass-action system associated with the chemical reaction network is given by 
\begin{equation}\label{eq:archetypal-ODE}
    \begin{aligned}
        \dfrac{dX(t)}{dt} & = - \alpha X(t)Y(t) + \beta Y(t),\\
        \dfrac{dY(t)}{dt} & =  \alpha X(t)Y(t) - \beta Y(t),\\
    \end{aligned}
\end{equation}
where $X(t)$ and $Y(t)$ represent the concentrations of the species $X$ and $Y$ at time $t$, respectively, with a slight abuse of notation. 

\section{Main results}\label{sec:main}
In this section, we present our main results. We first demonstrate that ACR and aACR are practically indistinguishable using a simple, representative example. We then show that aACR is ubiquitous in biologically relevant networks. 

\subsection{ACR and aACR are practically indistinguishable}

\begin{figure}[h]
    \centering
    \includegraphics[width=0.8\linewidth]{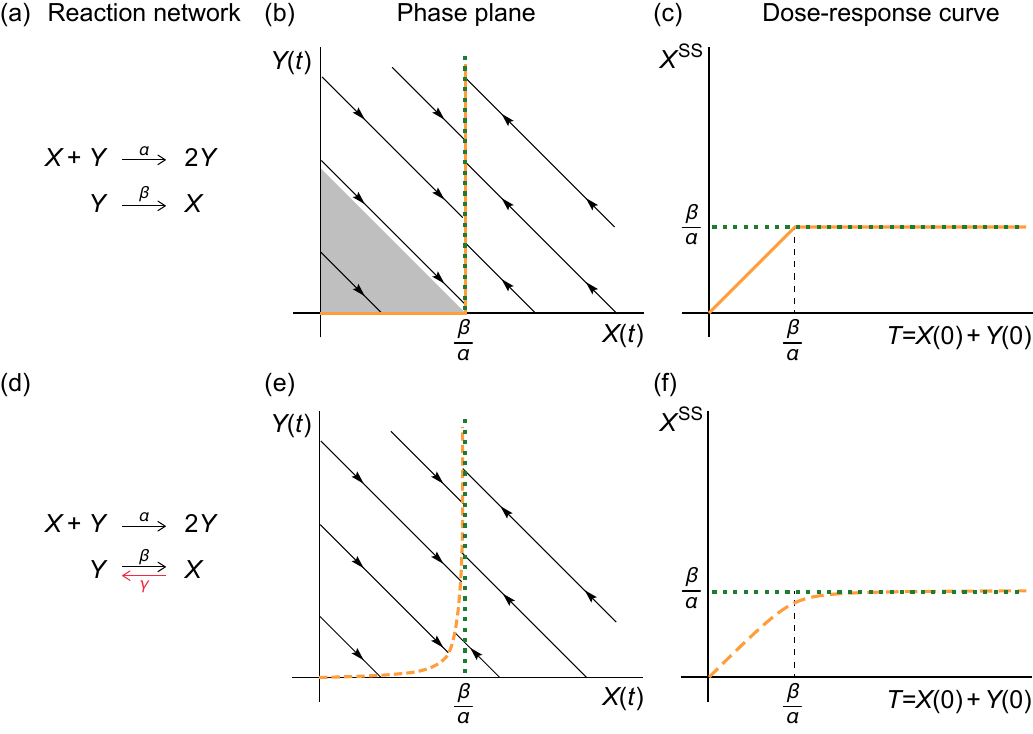}
    \caption{
    \textbf{Phase planes and dose-response curves of the reaction networks in \eqref{crn:archetypal} and \eqref{crn:archetypal-modified}. } 
    (\textbf{a}) The reaction network diagram given in~\eqref{crn:archetypal}.  
    (\textbf{b}) The phase plane describing the dynamics of the reaction network in \eqref{crn:archetypal}. Orange lines represent the set of steady states. 
    (\textbf{c}) The dose-response curve of the steady-state concentration of $X$, denoted by $X^{\text{SS}}$, with respect to the conserved quantity in the reaction network, $T=X(0)+Y(0)$. 
    (\textbf{d}) A reaction network modified from (a) by adding one more reaction, $X \to Y$, given in~\eqref{crn:archetypal-modified}. 
    (\textbf{e}) The phase-plane of the modified network. 
    (\textbf{f}) The dose-response curve the modified network. Although $X^{\text{SS}}$ no longer becomes a constant function, it converges to the same value. Notably, the two dose-response curves in (c) and (f) are nearly indistinguishable.}
    \label{fig:archetypal}
\end{figure}

To illustrate the classical concept of ACR and our newly proposed concept of aACR and their differences, we analyze the simple reaction network in~\eqref{crn:archetypal} (Fig.~\ref{fig:archetypal}a). 
The first reaction, $X+Y \to 2Y$, can be seen as an autoactivation reaction where an inactive form of an enzyme ($X$) is activated by interacting with an active form ($Y$). The second reaction, $Y \to X$, can be seen as an inactivation of the active form $Y$. 
Since there is no pure production or decay of the enzymes, the total concentration of enzymes is always a constant, i.e., $X(t)+Y(t) = X(0)+Y(0) = T$ for all $t \geq 0$. 
Dynamics of the concentrations $X(t)$ and $Y(t)$ can be described by the system of ODEs as in~\eqref{eq:archetypal-ODE}.


The time evolution of the concentrations can be depicted on the phase plane (Fig.~\ref{fig:archetypal}b). On the phase plane of this reaction network, every point with positive initial concentrations of both $X$ and $Y$ eventually converge to the set of steady states (orange lines; Fig.~\ref{fig:archetypal}b). Specifically, except for the small region of low total concentrations (gray region; Fig.~\ref{fig:archetypal}b), every initial point converges to a steady state lying on a vertical line (dotted vertical line; Fig.~\ref{fig:archetypal}b). That is, wherever an initial point is, its steady state has the same concentration of $X$. This implies that the steady-state concentration of $X$, denoted by $X^{\text{SS}}$, is independent of the overall supply of this system (i.e., the total enzyme concentration). When this independence appears, we say ``$X$ has ACR.'' This phenomenon of ACR is achieved in networks with specific structural properties revealed by Shinar and Feinberg \cite{ShinarFeinberg2010Science}. Consequently, the dose-response curve of $X^{\text{SS}}$ with respect to the total enzyme concentration, $T$, reaches a fixed value (i.e., the horizontal line) once $T$ exceeds a threshold (Fig.~\ref{fig:archetypal}c), which indicates that $X$ admits ACR.

While exhibiting ACR in a certain species serves as a valuable information indicating that the steady-state concentration of the species is completely independent of an initial condition, this property can easily be lost with slight modifications to the network structure. Here, we modify the network in Fig.~\ref{fig:archetypal}a by adding the reverse of the second reaction and get the following chemical reaction network (Fig.~\ref{fig:archetypal}d):

\begin{equation}\label{crn:archetypal-modified}
    \vcenter{\hbox{\begin{tikzpicture}
            \begin{scope}[shift={(0,0)}]
            \node (1a) at (0,0) [left]  {$\cf{X} + \cf{Y}$} ; 
            \node (1b) at (1.25,0) [right] {$2\,\cf{Y}$,}; 
            \draw[fwdrxn, transform canvas={yshift=0.00pt}] (1a)--(1b) node [midway, above] {$\alpha$};
            \end{scope}
            \begin{scope}[shift={(3,0)}]
        \node (2a) at (0,0) [left] {$\cf{Y}$}; 
        \node (2b) at (1.25,0) [right] {$\cf{X}$.};
        \draw[fwdrxn, transform canvas={yshift=1.75pt}] (2a)--(2b) node [midway, above] {$\beta$};
        \draw[fwdrxn, red, transform canvas={yshift=-1.75pt}] (2b)--(2a) node [midway, below] {$\mathcolor{red}{\gamma}$};
    \end{scope}
        \end{tikzpicture}}}
\end{equation}

We obtained the phase plane and dose-response curve for this modified network as well (Fig.~\ref{fig:archetypal}e,f). Due to the modification, the set of steady states and thus the robustness of $X^{\text{SS}}$ were altered. 
Specifically, the set of steady states is not a vertical line but rather an asymptotic curve (orange curve; Fig.~\ref{fig:archetypal}e) that approaches a vertical line (dotted vertical line; Fig.~\ref{fig:archetypal}e). 
Accordingly, the dose-response curve of $X^{\text{SS}}$ with respect to $T$ reaches a fixed value (i.e., the horizontal line) asymptotically as $T$ becomes large (Fig.~\ref{fig:archetypal}f).
This indicates that the steady-state concentration of $X$ is not exactly a constant, but it approaches a constant asymptotically as the total enzyme concentration increases. Although the steady-state concentration of $X$ is not completely independent of the total concentration $T$, {\em its dependence on $T$ is negligible when $T$ is large enough}. 
(Note that even when $X$ has ACR, $X^{\text{SS}}$ {\em does} depend on $T$ when $T$ is small; see Fig.~\ref{fig:archetypal}c.)
Therefore, we say ``$X$ has \textit{aACR} with respect to $T$'' when this asymptotic independence appears (see Definition~\ref{def:aACR} for a precise definition).

Although the two networks do not show  exactly the same dose-response curves (Fig.~\ref{fig:archetypal}c,f), it is noteworthy that there is no meaningful difference for practical experimental purposes. In particular, as illustrated in Fig.~\ref{fig:dose-response}, a dose-response curve is commonly measured with noise and limited data resolution. Thus, it is nearly impossible to determine whether an observed flat dose-response curve with discrete data points originates from a perfectly horizontal line with ACR (Fig.~\ref{fig:archetypal}c) or a nearly horizontal plateau with aACR (Fig.~\ref{fig:archetypal}f). This means that identifying networks showing aACR is as important as identifying networks with ACR. Moreover, the aACR of $X$ arises solely from the network structure, independently of the rate constant $\gamma$ associated with the added reverse reaction, $Y \leftarrow X $. 

In what follows, we calculate steady-state solutions for the examples in~\eqref{crn:archetypal} and~\eqref{crn:archetypal-modified} (Fig.~\ref{fig:archetypal}a,d). More detailed derivations are given in \cref{subsec:model-derivation}. 
First, the archetypal network in~\eqref{crn:archetypal} has the following steady-state equation 
%
\begin{equation}\label{eq:archetypal-SS-1}
    \begin{aligned}
        0 & = - \alpha X^{\text{SS}}Y^{\text{SS}} + \beta Y^{\text{SS}},\\
        0 & =  \alpha X^{\text{SS}}Y^{\text{SS}} - \beta Y^{\text{SS}},\\
    \end{aligned}
\end{equation}
because the steady state is described by setting the time derivatives in~\eqref{eq:archetypal-ODE} zero.  With the conservation law $T = X(t) + Y(t) = X(0) + Y(0)$ of the system, the positive steady states of this system are given by (Fig.~\ref{fig:archetypal}b):
\begin{equation}\label{eq:archetypal-SS-2}
    (X^{\text{SS}}, Y^{\text{SS}}) = \begin{cases}
        \left(\dfrac{\beta}{\alpha}, T - \dfrac{\beta}{\alpha}\right), & \text{for } \,  T\geq \dfrac{\beta}{\alpha}, \\ 
        (T,0), & \text{for } \,  T < \dfrac{\beta}{\alpha}.
    \end{cases}
\end{equation}

Now, let us focus on the modified network shown in \eqref{crn:archetypal-modified} (Fig.~\ref{fig:archetypal}d). The corresponding mass-action system is the following system of ODEs:
\begin{equation}
    \begin{aligned}
        \dfrac{dX(t)}{dt} & = - \alpha X(t)Y(t) + \beta Y(t) - \gamma X(t),\\
        \dfrac{dY(t)}{dt} & =  \alpha X(t)Y(t) - \beta Y(t) + \gamma X(t).\\
    \end{aligned}
\end{equation}
It is still true that we have $T = X(t) + Y(t) = X(0)+Y(0)$ constant. The steady state of $X$ in terms of $T$ then is given by (Fig.~\ref{fig:archetypal}e):
\begin{align}
    X^{\text{SS}} = \dfrac{1}{2\alpha} (\beta + \gamma + \alpha T - \sqrt{(\alpha T - \beta + \gamma)^2 + 4\beta\gamma})
\end{align}
which as $T$ goes to infinity has a limit of $\frac{\beta}{\alpha}$, which shows that the species $X$ exhibits aACR.


\subsection{aACR is ubiquitous in real biochemical systems, even without core ACR motifs}
In the examples shown in Fig.~\ref{fig:archetypal}, both ACR and aACR were analyzed in terms of variations in the total concentration \( T \), which is the only conserved quantity in the system. However, a subtle yet important distinction between the two concepts arises when considering systems with multiple conserved quantities. 
The conventional notion of ACR refers to the property by which a species maintains exactly the same steady-state concentration across \textit{all} initial conditions, regardless of variations in the conserved quantities \cite{ShinarFeinberg2010Science}. 
In contrast, aACR must be analyzed by varying each conserved quantity individually, as the steady-state behavior of a system may differ depending on which quantity is perturbed. This distinction did not arise in Fig.~\ref{fig:archetypal}, where only a single conserved quantity was present, but it becomes critical in systems such as the EnvZ-OmpR network (Fig.~\ref{fig:biochem-example}), which involves two conserved quantities: the total concentrations of \( X \)-containing and \( Y \)-containing species.

\begin{figure}[h]
    \centering
    \includegraphics[width=0.8\linewidth]{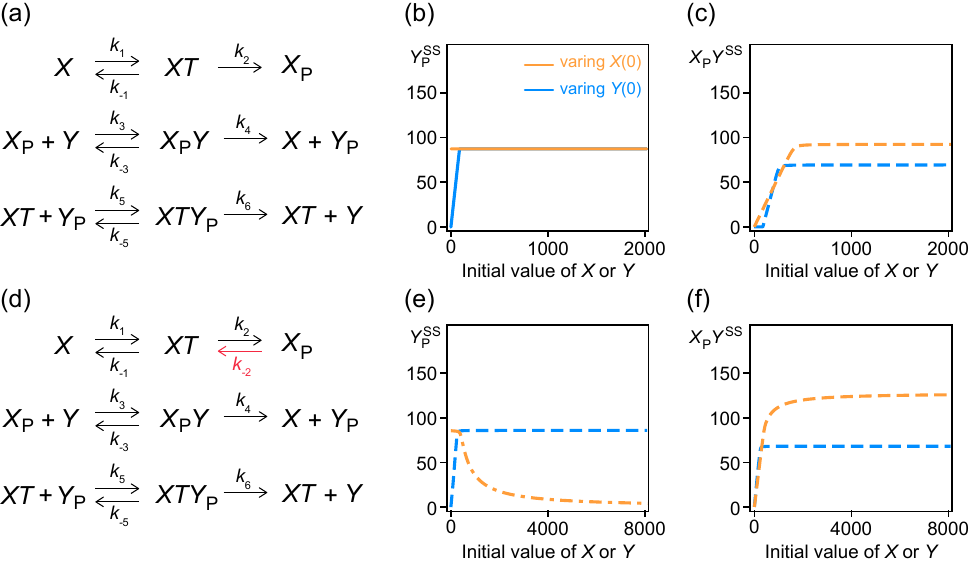}
    \caption{\textbf{Biologically relevant networks showing ACR and aACR phenomena.} 
    (\textbf{a}) The EnvZ-OmpR signaling network, a well-studied biochemical system, serves as a key example of a network with ACR in species $Y_\text{P}$.
    (\textbf{b}) Dose-response curves for $Y_\text{P}$ show that its steady-state concentration reaches a fixed value (horizontal line) when varying the initial concentration of $X$ (orange curve) or $Y$ (blue curve), demonstrating ACR. 
    Here, $Y_\text{P}^{\text{SS}}$ denotes the steady-state concentration of $Y_\text{P}$. Note that these two curves overlap as they reach the same $Y_\text{P}^{\text{SS}}$ value.
    (\textbf{c}) In contrast, $X_\text{P}Y$ exhibits aACR, denoted by dashed lines. This is demonstrated by the dose-response curve of its steady-state concentration, denoted by $X_\text{P}Y^{\text{SS}}$, approaching but never precisely reaching a fixed value as the initial concentrations of $X$ or $Y$ are increasing.
    (\textbf{d}) A network modified from the network in (a) by making one irreversible reaction, $XT \to X_\text{P}$, reversible.
    (\textbf{e}) In this modified network, $Y_\text{P}$ no longer exhibits ACR: its steady-state concentration, $Y_\text{P}^{\text{SS}}$, approaches zero when varying $X(0)$ and shows aACR when varying $Y(0)$.
    (\textbf{f}) Despite this modification, $X_\text{P}Y$ retains aACR, highlighting the robustness of this property compared to ACR.}
    \label{fig:biochem-example}
\end{figure}

To better align with experimental practices \cite{haas2012vivo} as illustrated in Fig.~\ref{fig:dose-response}, we henceforth define aACR \textit{with respect to a specific species} (see Definition~\ref{def:aACR} for a precise definition). Specifically, the steady-state concentration of an output species is analyzed as the initial concentration of a particular input species is varied, while keeping all other initial concentrations fixed.

The EnvZ-OmpR osmoregulation system in \textit{E. coli} is a well-studied example of a two-component signaling network (Fig.~\ref{fig:biochem-example}a), where robustness in molecular concentrations plays a critical role in cellular function.
This system consists of the sensor kinase EnvZ (denoted as \( X \)) and the response regulator OmpR (denoted as \( Y \)). Both proteins can exist in phosphorylated forms, \( X_\text{P} \) and \( Y_\text{P} \), which are central to the signaling process. The kinase \( X \) autophosphorylates using adenosine triphosphate (ATP) as a phosphate donor, and \( X_\text{P} \) subsequently transfers the phosphoryl group to \( Y \), forming \( Y_\text{P} \). Dephosphorylation of \( Y_\text{P} \) is catalyzed by \( X \) in the presence of ATP or adenosine diphosphate (ADP). The phosphorylated response regulator \( Y_\text{P} \) acts as a transcription factor, regulating the expression of genes involved in osmolarity adaptation, making its concentration crucial for proper cellular function.  

Importantly, the network conserves the total concentrations of \( X \)-containing \allowbreak species, \( T_X = X(t) + XT(t) + X_\text{P}(t) + X_\text{P}Y(t) + XTY_\text{P}(t) \), and that of \( Y \)-containing species, \( T_Y = Y(t) + Y_\text{P}(t) + X_\text{P}Y(t) + XTY_\text{P}(t) \). 
Previous studies have shown that this network exhibits ACR in \( Y_\text{P} \), meaning that the steady-state concentration of \( Y_\text{P} \) is independent of initial conditions \( X(0) \) and \( Y(0) \) and, equivalently, of the total concentrations \( T_X \) and \( T_Y \) \cite{ShinarFeinberg2010Science}. Accordingly, the dose-response curves of the steady-state concentration of $Y_\text{P}$, denoted by $Y_\text{P}^{\text{SS}}$, with respect to varying initial conditions $X(0)$ and $Y(0)$ form perfect horizontal lines (Fig.~\ref{fig:biochem-example}b). 
In contrast, we found that other species in the network, such as \( X_\text{P}Y \), exhibit aACR, as illustrated by the dose-response curves that asymptotically approach a fixed value as \( X(0) \) or \( Y(0) \) increases (Fig.~\ref{fig:biochem-example}c). This indicates that aACR is already present in this network, although it had not been previously recognized. These findings highlight the importance of considering aACR when studying biological robustness, as it may be more prevalent than ACR in real-world systems.

Since the condition for a reaction network to achieve ACR revealed in \cite{ShinarFeinberg2010Science} is sensitive to network structure, we investigated how vulnerable ACR is to slight structural modifications. To this end, we made a minimal change to the EnvZ-OmpR network by adding a reverse reaction to one of its irreversible reactions, $XT \to X_\text{P}$ (Fig.~\ref{fig:biochem-example}d). 
In the modified network, the behavior of $Y_\text{P}$ changes dramatically (Fig.~\ref{fig:biochem-example}e). While it previously exhibited ACR, it now either approaches zero when varying $X(0)$ or shows aACR when varying $Y(0)$ (orange and blue curves in Fig.~\ref{fig:biochem-example}e). This demonstrates how sensitive ACR can be to even minor structural modifications. In contrast, the concentration of $X_\text{P}Y$ retains its aACR behavior (Fig.~\ref{fig:biochem-example}f), indicating the persistence of this property.

To systematically explore these observations, we analyzed the behavior of all species in the network with respect to varying the initial concentration of each species. Tables~\ref{tab:SF-2B-unmodified} and \ref{tab:SF-2B-modified-1} summarize these results for the unmodified and modified networks, respectively. The possible behaviors are categorized as: ACR, aACR, divergence (denoted by $\nearrow \infty$), or extinction (i.e., approaching 0, denoted by $\searrow0$). 

In the unmodified network (Fig.~\ref{fig:biochem-example}a), ACR is observed only in $Y_\text{P}$ as expected, while aACR is pervasive across several other species, including $X_\text{P}Y$ (Table~\ref{tab:SF-2B-unmodified}). In the modified network (Fig.~\ref{fig:biochem-example}d), all instances of ACR disappeared, with some transitioning to aACR and extinction (Table~\ref{tab:SF-2B-modified-1}). Notably, most instances of aACR persist, verifying its prevalence and relative robustness compared to ACR. 
These characteristics of aACR can repeatedly be observed in many other modified networks and even other biological systems, including a phospho-dephosphorylation futile cycle (Fig.~\ref{Supp-fig:additional-networks}, Tables~\ref{Supp-tab:FutCyc-unmodified} and \ref{Supp-tab:FutCyc-modified-3})
\cite{ShinarFeinberg2010Science, conradi2019multistationarity, goldbeter1981amplified, kholodenko2000negative}.

Furthermore, we mathematically proved that the widespread emergence of aACR arises from intrinsic features of network structures, independent of parameter fine-tuning (Theorems~\ref{thm:main_theorem} and \ref{thm:second_theorem}). In particular, any system with at least one positive conserved quantity is guaranteed to exhibit aACR. Finally, we developed an algebraic method to identify all entries in Tables~\ref{tab:SF-2B-unmodified} and \ref{tab:SF-2B-modified-1}, without relying on numerical simulations (see \cref{sec:application-characterization} for details). 

\begin{table}[h]
    \centering
    \caption{\textbf{Summary of steady-state concentration responses of single species with respect to a increasing initial value of another species for the network in Fig.~\ref{fig:biochem-example}a.}
    The four possible behaviors, ACR, aACR, divergence, or extinction (i.e., approaching 0), are denoted by ACR, aACR, $\nearrow\infty$, and $\searrow0$ in the table. Although the table was obtained from numerical simulation results, all entries can also be identified using a method based on computational algebraic geometry. See Supplementary Text for a detailed description about the method.}
    \begin{tabular}{c|c|c|c|c|c|c|c|}
         & $X^{\text{SS}}$ & $XT^{\text{SS}}$ & $X_\text{P}^{\text{SS}}$ & $Y^{\text{SS}}$ & $Y_\text{P}^{\text{SS}}$ & $X_\text{P}Y^{\text{SS}}$ & $XTY_\text{P}^{\text{SS}}$ \\ \hline
        $X(0)$ & \textcolor{newred1}{aACR} & \textcolor{newred1}{aACR} & $\nearrow\infty$ & $\searrow$0 &\textcolor{newblue1}{ACR}& \textcolor{newred1}{aACR} & \textcolor{newred1}{aACR} \\ \hline
        $XT(0)$ & \textcolor{newred1}{aACR} & \textcolor{newred1}{aACR} & $\nearrow\infty$ & $\searrow$0 &\textcolor{newblue1}{ACR}& \textcolor{newred1}{aACR} & \textcolor{newred1}{aACR} \\ \hline
        $X_\text{P}(0)$ & \textcolor{newred1}{aACR} & \textcolor{newred1}{aACR} & $\nearrow\infty$ & $\searrow$0 &\textcolor{newblue1}{ACR}& \textcolor{newred1}{aACR} & \textcolor{newred1}{aACR} \\ \hline
        $Y(0)$ & \textcolor{newred1}{aACR} & \textcolor{newred1}{aACR} & $\searrow$0 & $\nearrow\infty$ &\textcolor{newblue1}{ACR}& \textcolor{newred1}{aACR} & \textcolor{newred1}{aACR} \\ \hline
        $Y_\text{P}(0)$ & \textcolor{newred1}{aACR} & \textcolor{newred1}{aACR} & $\searrow$0 & $\nearrow\infty$ &\textcolor{newblue1}{ACR}& \textcolor{newred1}{aACR} & \textcolor{newred1}{aACR} \\ \hline
        $X_\text{P}Y(0)$ & $\nearrow\infty$ & $\nearrow\infty$ & \textcolor{newred1}{aACR} & $\nearrow\infty$ &\textcolor{newblue1}{ACR}& $\nearrow\infty$ & $\nearrow\infty$ \\ \hline
        $XTY_\text{P}(0)$ & $\nearrow\infty$ & $\nearrow\infty$ & \textcolor{newred1}{aACR} & $\nearrow\infty$ &\textcolor{newblue1}{ACR}& $\nearrow\infty$ & $\nearrow\infty$ \\ \hline
    \end{tabular}
    \label{tab:SF-2B-unmodified}
\end{table}

\begin{table}[h]
    \centering
    \caption{
    \textbf{Summary of steady-state concentration responses for the network in Fig.~\ref{fig:biochem-example}d.}
    The yellow-colored cells show where the steady-state behavior has changed when compared with the unmodified network, while transparent cells have the same steady-state behaviors as the unmodified network (see Table~\ref{tab:SF-2B-unmodified}.) A majority of the steady-state behaviors are maintained after this modification.}
    \begin{tabular}{c|c|c|c|c|c|c|c|}
         & $X^{\text{SS}}$ & $XT^{\text{SS}}$ & $X_\text{P}^{\text{SS}}$ & $Y^{\text{SS}}$ & $Y_\text{P}^{\text{SS}}$ & $X_\text{P}Y^{\text{SS}}$ & $XTY_\text{P}^{\text{SS}}$ \\ \hline
        $X(0)$ & \cellcolor{yellow!30} $\nearrow\infty$ & \cellcolor{yellow!30}$\nearrow\infty$ & $\nearrow\infty$ & $\searrow$0 & \cellcolor{yellow!30}$\searrow$0 & \textcolor{newred1}{aACR} & \textcolor{newred1}{aACR} \\ \hline
        $XT(0)$  & \cellcolor{yellow!30}$\nearrow\infty$ & \cellcolor{yellow!30}$\nearrow\infty$ & $\nearrow\infty$ & $\searrow$0 & \cellcolor{yellow!30}$\searrow$0 & \textcolor{newred1}{aACR} & \textcolor{newred1}{aACR} \\ \hline
        $X_\text{P}(0)$  & \cellcolor{yellow!30}$\nearrow\infty$ & \cellcolor{yellow!30}$\nearrow\infty$ & $\nearrow\infty$ & $\searrow$0 & \cellcolor{yellow!30}$\searrow$0 & \textcolor{newred1}{aACR} & \textcolor{newred1}{aACR} \\ \hline
        $Y(0)$ & \textcolor{newred1}{aACR} & \textcolor{newred1}{aACR} & $\searrow$0 & $\nearrow\infty$ &\cellcolor{yellow!30} \textcolor{newred1}{aACR} & \textcolor{newred1}{aACR} & \textcolor{newred1}{aACR} \\ \hline
        $Y_\text{P}(0)$ & \textcolor{newred1}{aACR} & \textcolor{newred1}{aACR} & $\searrow$0 & $\nearrow\infty$ & \cellcolor{yellow!30} \textcolor{newred1}{aACR} & \textcolor{newred1}{aACR} & \textcolor{newred1}{aACR} \\ \hline
        $X_\text{P}Y(0)$ & $\nearrow\infty$ & $\nearrow\infty$ & \textcolor{newred1}{aACR} & $\nearrow\infty$ &\cellcolor{yellow!30} \textcolor{newred1}{aACR} & $\nearrow\infty$ & $\nearrow\infty$ \\ \hline
        $XTY_\text{P}(0)$ & $\nearrow\infty$ & $\nearrow\infty$ & \textcolor{newred1}{aACR} & $\nearrow\infty$ &\cellcolor{yellow!30} \textcolor{newred1}{aACR} & $\nearrow\infty$ & $\nearrow\infty$ \\ \hline
    \end{tabular}
    \label{tab:SF-2B-modified-1}
\end{table}

\section{Mathematical analysis ensuring the ubiquity of aACR phenomena}\label{sec:math-analysis}
In this section, we present the mathematical definitions, rigorous theorems, and proofs that form the foundation of the results described above. We also illustrate, through representative examples, how the algebraic methods developed in the course of these proofs enable the complete characterization summarized in Tables~\ref{tab:SF-2B-unmodified} and \ref{tab:SF-2B-modified-1}.

\subsection{Mathematical preliminaries and notation}
\paragraph{Notational remark.} Here, we use both $A(t)$ and $[A](t)$ to denote the concentration of species $A$ at time $t$. The time argument ``$(t)$'' is often omitted when no confusion arises. Square brackets are employed to distinguish between expressions like $[A][B]$ and $[AB]$ in the presence of species $A$, $B$, and their binding complex $AB$. For notational convenience, we use both $\overline{A}$ (or $[\overline{A}]$) and $A^{\text{SS}}$ to represent the steady-state concentration of species $A$.


\begin{definition}\label{def:dose-response}
    We say that a reaction system  {\em has well defined dose-response curves for input $X_i$} if 
    there exists a unique positive equilibrium within the stoichiometric compatibility class that contains the initial concentration $\mathbf{x}_0 + \lambda \mathbf{e}_i$ for all $\lambda$ large enough, where $\mathbf{e}_i$ is the $i^{th}$ element of the standard basis of $\mathbb{R}^n$. The function that maps the number $\lambda$ to the $j^{th}$ coordinate of the equilibrium vector within the stoichiometric compatibility class of $\mathbf{x}_0 + \lambda \mathbf{e}_i$ is called the {\em dose-response curve of $X_j$ with respect to $X_i$}. That is, this function maps $\lambda$ to $\overline{X}_j(\lambda)$.
\end{definition}

\begin{definition}\label{def:aACR} Consider a reaction system that  has well defined dose-response curves for input $X_i$. 
A species $X_j$ is said to \emph{admit asymptotic ACR (aACR) with respect to $X_i$} if there exists $a_{i,j} > 0$ such that for any $\mathbf{x}_0 \in \mathbb{R}_{>0}^n$, we have
    \begin{equation}
        \lim_{\lambda \to \infty} \overline{X}_j(\lambda) = a_{i,j},
    \end{equation}
    where $\overline{X}_j(\lambda)$ is the steady state of species $X_j$ for the system with initial concentration $\mathbf{x}_0 + \lambda \mathbf{e}_i$. In other words, we say $X_i$ admits aACR if the dose-response curve of $X_j$ with respect to $X_i$ converges to a positive real number as $\lambda \to \infty$.
\end{definition}

We present a theorem ensuring the ubiquity of aACR and another theorem providing a tool for analyzing steady-state behavior of individual species in a biochemical system. Consider a reaction system given by
\begin{align}\label{eq:CRN}
    \dfrac{dX_1}{dt}&= f_1(X_1,\dots,X_n)\nonumber\\
    \vdots\\
    \dfrac{dX_n}{dt} &= f_n(X_1,\dots,X_n)\nonumber
\end{align}
where $f_i(X_1,\dots,X_n)$ are rational functions, with conservation laws
\begin{align}\label{eq:cons-law}
    \sum_{i=1}^n\alpha_i^{(1)}X_i(t) &= \sum_{i=1}^n\alpha_i^{(1)}X_i(0)\nonumber\\
    \vdots\\
    \sum_{i=1}^n\alpha_i^{(m)}X_i(t) &= \sum_{i=1}^n\alpha_i^{(m)}X_i(0)\nonumber
\end{align}

\subsection{Theorem statements and proofs}

Our main theorem (Theorem~\ref{thm:main_theorem}) essentially states that any positive conservation law with proper support can be used to guarantee the following: for any species $X_i$ that {\em does not} show up in the conservation law, there exists at least one species $X_j$ that {\em does} show up in it and admits aACR with respect to $X_i$. Here, a positive conservation law means its associated coefficients $\alpha_i$'s are all nonnegative (see \eqref{eq:cons-law}). 

\begin{theorem}\label{thm:main_theorem}
    Consider a reaction system given by \eqref{eq:CRN} that has well defined dose-response curves for input $X_i$. Assume that this system  admits a positive conservation law with support set $\Sigma = \{X_{j_1},\dots,X_{j_k}\}$, such that $X_i$ is not in $\Sigma$.
    Then there exists at least one species $X_j$ within $\Sigma$ such that $X_j$ admits aACR with respect to $X_i$.
\end{theorem}
In order to prove this theorem, we first discuss the following useful lemma. The proof of Lemma~\ref{lem:LimitExists} follows its statement below, the proof of Theorem~\ref{thm:main_theorem} is in Appendix. 

\begin{lemma}\label{lem:LimitExists}
    Let $P(x,\lambda)$ be a polynomial with real variables $x$ and $\lambda$, and consider a continuous function $x(\lambda)$  defined for all sufficiently large $\lambda > \lambda_0$, satisfying $P(x(\lambda),\lambda)=0$. Then $\displaystyle \lim_{\lambda \to \infty} x(\lambda)$ exists (within the closed interval $[-\infty, \infty]$).

    Moreover, let us  write 
    \begin{equation}
        P(x,\lambda) = \lambda^m q(x) + r(x,\lambda),
    \end{equation}
    where $m$ is the highest degree of $\lambda$ in $P(x,\lambda)$, $q(x)$ is non-zero, and the degree of $\lambda$ in $r(x,\lambda)$ is less than $m$. Then, if the limit is finite, it is a root of $q(x)$.
\end{lemma}

\begin{proof}[Proof of Lemma~\ref{lem:LimitExists}]
    To show that the limit exists it is enough to prove that
    \begin{equation}
        \liminf_{\lambda\to\infty} x(\lambda) = \limsup_{\lambda\to\infty} x(\lambda).
    \end{equation}
    Write $P(x,\lambda)$ as
    \begin{equation}
        P(x,\lambda) = \lambda^m q(x) + r(x,\lambda),
    \end{equation}
    where $m$ is the highest degree of $\lambda$ in $P(x,\lambda)$, $q(x)$ is non-zero, and the degree of $\lambda$ in $r(x,\lambda)$ is less than $m$. 
    
    If $m=0$, then $P(x,\lambda)=P(x)=q(x)$, which have only finitely many roots. Since $x(\lambda)$ is continuous and must admit a value which is a root of $P(x)=0$, it must be a constant function. Thus in this case the limit exists and it is a root of $q(x)$.
    
    We now assume that $m>0$. Then, 
    \begin{align}
        \lambda^m q(x(\lambda)) + r(x(\lambda),\lambda) &= 0,\\
        \Rightarrow q(x(\lambda)) + \dfrac{1}{\lambda^m} r(x(\lambda),\lambda) &= 0.
    \end{align}
    Suppose by contradiction that 
    $\liminf_{\lambda\to\infty}x(\lambda) = c_1$ is strictly less than \\
    $\limsup_{\lambda\to\infty}x(\lambda) = c_2$, with $c_1,c_2\in[-\infty,\infty]$. Then, for all $c \in (c_1,c_2)$, there exists a sequence $\lambda_k \to \infty$ such that
    \begin{equation}
        \lim_{k\to\infty} x(\lambda_k) = c
    \end{equation}
    by the Intermediate Value Theorem. 
    Then, in the limit $\dfrac{1}{\lambda^m} r(x(\lambda),\lambda) = 0$ because $x(\lambda)$ is bounded.
    But then, as $q(x)$ is a polynomial, it is continuous, so 
    \begin{equation}
        \lim_{k\to\infty} q(x(\lambda_k)) = q(c) = 0,
    \end{equation}
    which implies $c$ is a root of $q$ for all $c\in(c_1,c_2)$. Since the non-zero polynomial $q(x)$ can have only finitely many roots, it leads to a contradiction. Thus, the limit exists, and if it is finite, it must be a root of $q(x)$.
\end{proof}

\begin{remark}\label{rmk:root-of-q}
    Recall from the proof of Lemma~\ref{lem:LimitExists} that any finite limit of the dose-response curve of $X_j$ with respect to $X_i$ must be a root of the polynomial $q(x)$ introduced there with respect to $P(\overline{X}_j(\lambda),\lambda)$ as in the proof of Theorem~\ref{thm:main_theorem}. From this property we see the following:
    
    If the polynomial $q(x)$ for $X_j$ does not have any non-negative roots, then we must have
    \begin{equation}
        \lim_{\lambda\to\infty} \overline{X}_j(\lambda) =\infty.
    \end{equation}
    Similarly, if $q(x)$ for $X_j$ does not have $x=0$ as a root (which is often the case), then
    \begin{equation}
        0< \lim_{\lambda\to\infty} \overline{X}_j(\lambda) \leq \infty.
    \end{equation}
    Also note that if $X_j$ appears in the support of a positive conservation law, then it is guaranteed to have a finite limit:
    \begin{equation}
        0\leq\lim_{\lambda\to\infty}\overline{X}_j(\lambda)<\infty.
    \end{equation}
    Furthermore, if in this case $P(\overline{X}_j(\lambda),\lambda)$ does not depend on $\lambda$, then the dose-response curve of $X_j$ with respect to $X_i$ is eventually constant. In other words, we have not only the following limit:
    \begin{equation}
        \lim_{\lambda \to \infty} \overline{X}_j(\lambda) = c \in (0,\infty),
    \end{equation}
    but also that $\overline{X}_j(\lambda) = c$ for \emph{all} sufficiently large $\lambda>0$.
    
    These observations together are particularly useful when examining whether a given species $X_j$ exhibits aACR with respect to $X_i$, as the following theorem summarizes.
\end{remark}

\begin{theorem} \label{thm:second_theorem}
    Consider a reaction system given by \eqref{eq:CRN} that has well defined dose-response curves for input $X_i$. Assume that this system  admits a positive conservation law with support set $\Sigma = \{X_{j_1},\dots,X_{j_k}\}$, such that $X_i$ is not in $\Sigma$. Assume also that $X_j$ is in $\Sigma$, and that the polynomial $q(x)$ for $X_j$ does not have a root at zero.
    Then $X_j$ has aACR with respect to $X_i$.
\end{theorem}

The proof of Theorem~\ref{thm:second_theorem} is in \cref{appendix:proofs}. 

\begin{remark}
    Note that the polynomial $P(x,\lambda)$ for $X_j$ can be found using Elimination Theory from computational algebraic geometry. In particular, we can use Gr\"{o}bner basis or resultants to get an explicit expression for the polynomial. This calculation can be simplified in the case of the original network having intermediates, as we can reduce the Gr\"{o}bner basis calculation to a calulation for the reduced network with no intermediates \cite{Sadeghimanesh2019}. 
\end{remark} 

\begin{remark}
   In the case where we have an explicit (rational) parametrization of the steady-state variety in terms of a species $X_j$ and other species we can greatly simplify the calculations for the polynomial $P(x,\lambda)$. This is quite a common situation and there exist many tools available to check for such a parametrization, for example, in the case of toric steady-state varieties, also called quasi-thermostatic \cite{Craciun2009,PerezMillan2011}, or for translated networks in the generalized mass-action setting \cite{johnston2014translated,Johnston2018}. 
 If such a parametrization exists, we can use it to reduce the system with $n$ polynomials, $n - s$ linear equations (conservation laws) and $2n - s$ variables ($n$ species and $n-s$ conserved quantities), to a system with just $n-s$ polynomial equations and just $2n-2s$ variables ($n-s$ species and $n-s$ conserved quantities). 
 From this reduced system, we can calculate via Gr\"{o}bner basis or other methods a polynomial in terms of $X_j$ and all conserved quantities, which can be used to perform all necessary calculations. We illustrate this principle in \cref{sec:application-characterization} for the reaction network in Fig.~\ref{fig:biochem-example}a. Note that in practice the number of conservation laws, $n-s$, is small. In Example~\ref{ex:analysis-biochem-unmod}, we have $n=7$ and $n-s=2$, so we reduce from a system of $9$ equations in $7$ variables, to a system of $2$ equations in $2$ variables.
\end{remark}

\subsection{Application of Theorem~\ref{thm:main_theorem}}
Theorem~\ref{thm:main_theorem} guarantees the existence of aACR under the presence of positive conservation laws, without the need of a detailed analysis. 
In this subsection, we demonstrate that how Theorem~\ref{thm:main_theorem} is applied to \textit{E. coli} EnvZ-OmpR system (Fig.~\ref{fig:biochem-example}). 

\begin{example}
Consider the reaction network modeling the \textit{E. coli} EnvZ-OmpR system from Fig.~\ref{fig:biochem-example}a in our main text. This system has two conserved quantities:
\begin{itemize}
    \item The total mass of $X$-related species: $[X] + [XT] + [X_\text{P}] + [X_\text{P}Y] + [XTY_\text{P}]= T_1.$
    \item The total mass of $Y$-related species: $[Y] + [Y_\text{P}] + [X_\text{P}Y] + [XTY_\text{P}]=T_2.$
\end{itemize}
Note that these give two positive conservation laws with supports \\
$\Sigma_{T_1}\coloneqq \{X, XT, X_\text{P}, X_\text{P}Y, XTY_\text{P}\}$ and $\Sigma_{T_2}\coloneqq \{Y, Y_\text{P}, X_\text{P}Y, XTY_\text{P}\}$, respectively. By Theorem~\ref{thm:main_theorem} we see that:
\begin{itemize}
    \item For each $X_i\in\{Y, Y_\text{P}\}$, there is at least one \\
    $X_j\in \Sigma_{T_1} = \{X, XT, X_\text{P}, X_\text{P}Y, XTY_\text{P}\}$ that admits aACR with respect to $X_i$.
    \item For each $X_i\in\{X, XT, X_\text{P}\}$, there is at least one \\
    $X_j\in \Sigma_{T_2} = \{Y, Y_\text{P}, X_\text{P}Y, XTY_\text{P}\}$ that admits aACR with respect to $X_i$.
\end{itemize}
Note also that the conservation laws do not change under modifications that irreversible reactions are changed to reversible ones. 
This explains why aACR are mostly preserved even when we modify the network by changing irreversible reactions to reversible ones (Tables~\ref{tab:SF-2B-unmodified} and \ref{tab:SF-2B-modified-1}).
\end{example}

\subsection{Lemma~\ref{lem:LimitExists} and Theorem~\ref{thm:second_theorem} allow complete characterization of steady-state behaviors}\label{sec:application-characterization}

We can use Lemma~\ref{lem:LimitExists} and Theorem~\ref{thm:second_theorem} to provide a \emph{complete characterization} of steady-state behaviors. Even more, in the following examples whenever the limit behavior of a dose-response curve has aACR, \emph{we explicitly calculate its limit}.

\begin{example}\label{ex:analysis-biochem-unmod}
The system of ODEs corresponding to the reaction network in Fig.~\ref{fig:biochem-example}a is given by 
\begin{align*}
    \dfrac{d[X]}{dt} &= -k_1[X] + k_{-1} [XT] + k_3 [X_\text{P}Y], \\ 
    \dfrac{d[XT]}{dt} &= k_1[X] - (k_{-1}+k_{-1})[XT] - k_{-3}[XT][Y_\text{P}] + (k_{-5}+k_4)[XTY_\text{P}], \\
    \dfrac{d[X_\text{P}]}{dt} &= k_{-1}[XT] - k_2[X_\text{P}][Y] + k_{-3}[X_\text{P}Y], \\ 
    \dfrac{d[Y]}{dt} &= k_4[XTY_\text{P}] - k_2[X_\text{P}][Y] + k_{-3}[X_\text{P}Y], \\
    \dfrac{d[Y_\text{P}]}{dt} &= k_3[X_\text{P}Y] - k_{-3}[XT][Y_\text{P}] + k_{-5}[XTY_\text{P}],\\ 
    \dfrac{d[X_\text{P}Y]}{dt} &= k_2[X_\text{P}][Y] - (k_{-3}+k_3)[X_\text{P}Y], \\ 
    \dfrac{d[XTY_\text{P}]}{dt} &= k_{-3}[XT][Y_\text{P}] - (k_{-5}+k_4)[XTY_\text{P}].
\end{align*}




The steady-state variety formula in terms of $[\overline{X_\text{P}Y}]$ and $[\overline{Y}]$ is given by: 
\begin{align}
    [\overline{X}] &= \frac{(k_{-1}+k_2)k_4[\overline{X_\text{P}Y}]}{k_1k_2} \label{eq:biochemex1-ss-for-X}\\ 
    [\overline{XT}] &= \frac{k_4[\overline{X_\text{P}Y}]}{k_2} \\ 
    [\overline{Y_\text{P}}] &= \frac{k_2(k_{-5}+k_6)}{k_5k_6} \\ 
    [\overline{X_\text{P}}] &= \frac{(k_{-3}+k_4)[\overline{X_\text{P}Y}]}{k_3[\overline{Y}]} \\ 
    [\overline{XTY_\text{P}}] &= \frac{k_4[\overline{X_\text{P}Y}]}{k_6}.\label{eq:biochemex1-ss-for-XTY_P}
\end{align}

\textbf{Note that $[\overline{Y_\text{P}}]$ is independent to $[\overline{XT}]$ and $[\overline{X_\text{P}}]$, which indicates the independence to initial conditions.}
While this formula is the function of $[\overline{X_\text{P}Y}]$ and $[\overline{Y}]$ in addition to rate constants, we can obtain the formula in terms of $T_1$ and $T_2$, which is more natural. Specifically, from the two conserved quantities,
\begin{align*}
    T_1 & = [\overline{X}] + [\overline{XT}] + [\overline{X_\text{P}}] + [\overline{X_\text{P}Y}] + [\overline{XTY_\text{P}}]\nonumber \\
    & = \frac{(k_{-1}+k_2)k_4[\overline{X_\text{P}Y}]}{k_1k_2} + \frac{k_4[\overline{X_\text{P}Y}]}{k_2} + \frac{(k_{-3}+k_4)[\overline{X_\text{P}Y}]}{k_3[\overline{Y}]} + [\overline{X_\text{P}Y}] + \frac{k_4[\overline{X_\text{P}Y}]}{k_6}\nonumber\\
    & = \left\{ \frac{(k_{-1}+k_2)k_4}{k_1k_2} + \frac{k_4}{k_2} + 1 + \frac{k_4}{k_6}\right\}[\overline{X_\text{P}Y}] + \frac{(k_{-3}+k_4)}{k_3}\frac{[\overline{X_\text{P}Y}]}{[\overline{Y}]},\\ 
    T_2 & = [\overline{Y}] + [\overline{Y_\text{P}}] + [\overline{X_\text{P}Y}] + [\overline{XTY_\text{P}}]\nonumber \\
     & = [\overline{Y}] + \frac{k_2(k_{-5}+k_6)}{k_5k_6} + [\overline{X_\text{P}Y}] + \frac{k_4[\overline{X_\text{P}Y}]}{k_6}\nonumber\\
     & = \frac{k_2(k_{-5}+k_6)}{k_5k_6} + [\overline{Y}] + \left\{1 + \frac{k_4}{k_6}\right\}[\overline{X_\text{P}Y}].
\end{align*}

For simplicity of calculations, we rewrite the conservation expressions as
\begin{align}
    T_1 & = a [\overline{X_\text{P}Y}] + b\frac{[\overline{X_\text{P}Y}]}{[\overline{Y}]},\label{eq:biochem-T_1}\\ 
    T_2 & = c + d [\overline{X_\text{P}Y}] + [\overline{Y}],\label{eq:biochem-T_2}
\end{align}
where $a,b,c,d$ are constants in terms of $k_i$ and in particular $c$ is the ACR value for $Y_\text{P}$.

From this, we can reduce into a single polynomial equation in terms of $[\overline{X_\text{P}Y}]$ and $T_1,T_2$ to then calculate the limit when $X_\text{P}Y$ is chosen to be $X_j$. Indeed, from the second expression Eq.~\eqref{eq:biochem-T_2},
\begin{align}\label{eq:biochemex2}
    [\overline{Y}] = T_2 - c - d[\overline{X_\text{P}Y}].
\end{align}
After plugging it in the first expression ~\eqref{eq:biochem-T_1} and clearing denominators we get
\begin{align}
    a[\overline{X_\text{P}Y}](T_2 - c- d[\overline{X_\text{P}Y}]) - T_1(T_2 - c- d[\overline{X_\text{P}Y}]) + b[\overline{X_\text{P}Y}] = 0.
\end{align}
Therefore, if $X_i \in \{X,XT,X_{\text{P}}\}$, then $T_1 = C_1 + \lambda$, and we get a function $P(x, \lambda)$ for $X_\text{P}Y$ evaluated at $x=[\overline{X_\text{P}Y}](\lambda)$: 
\begin{align}
    P([\overline{X_\text{P}Y}](\lambda),\lambda) &= - \lambda(T_2 - c- d[\overline{X_\text{P}Y}](\lambda)) + a[\overline{X_\text{P}Y}](\lambda)(T_2 - c- d[\overline{X_\text{P}Y}](\lambda)) \notag
    \\
    & - C_1(T_2 - c- d[\overline{X_\text{P}Y}](\lambda)) + b[\overline{X_\text{P}Y}](\lambda).
\end{align}
Thus, 
\begin{align}
    q([\overline{X_\text{P}Y}](\lambda)) = -(T_2 - c- d[\overline{X_\text{P}Y}](\lambda)).
\end{align}
Applying Lemma~\ref{lem:LimitExists} and by Remark~\ref{rmk:root-of-q}, we see that the limit of $[\overline{X_\text{P}Y}](\lambda)$ must be a root of $q$, so
\begin{equation}
\lim_{\lambda\to\infty} [\overline{X_\text{P}Y}](\lambda)=\dfrac{T_2 - c}{d}.
\end{equation}
From \eqref{eq:biochemex2}, we see that $Y$ goes to 0. From \eqref{eq:biochemex1-ss-for-X}--\eqref{eq:biochemex1-ss-for-XTY_P}, it follows that $X, XT,$ and $ XTY_\text{P}$ have aACR, and $X_\text{P}$ diverges with respect to $X_i \in \{X,XT,X_{\text{P}}\}$.

Similarly, if $X_i \in \{Y,Y_{\text{P}}\}$,
\begin{align}
    q([\overline{X_\text{P}Y}](\lambda)) = a[\overline{X_\text{P}Y}](\lambda)-T_1.
\end{align}
Thus, 
\begin{equation}
\lim_{\lambda\to\infty} [\overline{X_\text{P}Y}](\lambda)=\dfrac{T_1 }{a}.
\end{equation}
By \eqref{eq:biochemex2}, we see that $Y$ diverges. From \eqref{eq:biochemex1-ss-for-X}--\eqref{eq:biochemex1-ss-for-XTY_P}, $X, XT, XTY_\text{P}$ also have aACR, and $X_\text{P}$ goes to 0 with respect to $X_i \in \{Y,Y_{\text{P}}\}$.

Finally, if $X_i \in \{X_\text{P}Y,XTY_{\text{P}}\}$, then $T_1 = C_1 + \lambda, T_2 = C_2 + \lambda$, and we get a function $P(x, \lambda)$ for $X_\text{P}Y$ evaluated at $x=\overline{X_\text{P}Y}(\lambda)$: 
\begin{align}
    P([\overline{X_\text{P}Y}](\lambda),\lambda) &= -\lambda^2 + \lambda(a[\overline{X_\text{P}Y}] + c + d[\overline{X_\text{P}Y}])\notag \\
    &-a[\overline{X_\text{P}Y}](T_2 - c- d[\overline{X_\text{P}Y}]) - T_1(T_2 - c- d[\overline{X_\text{P}Y}]) + b[\overline{X_\text{P}Y}].
\end{align}
Thus,
\begin{equation}
    q([\overline{X_\text{P}Y}](\lambda))=-1.
\end{equation}
Therefore, as $q$ does not have any positive roots, 
\begin{equation}
\lim_{\lambda\to\infty} [\overline{X_\text{P}Y}](\lambda)=\infty, 
\end{equation}
by Remark~\ref{rmk:root-of-q}. From \eqref{eq:biochemex1-ss-for-X}--\eqref{eq:biochemex1-ss-for-XTY_P} we still see that $X, XT, XTY_\text{P}$ also diverge, but we do not gain information about the behavior of $Y$ or $X_\text{P}$ with respect to $X_i \in \{X_\text{P}Y,XTY_{\text{P}}\}$. Instead, we need to recalculate the polynomial for a different choice of $X_j$.
%

Let $X_j=Y$. Then from the second conservation law \eqref{eq:biochem-T_2}, we see that
\begin{align}
    [\overline{X_\text{P}Y}] = \frac{1}{d}(T_2 - c - [\overline{Y}]),
\end{align}
and after substituting it in \eqref{eq:biochem-T_1} and clearing denominators we get
\begin{align}
    \frac{a}{d}(T_2 - c - [\overline{Y}])[\overline{Y}] + \frac{b}{d}(T_2 - c - [\overline{Y}]) - T_1[\overline{Y}] = 0.
\end{align}
Therefore, if $X_i \in \{X_\text{P}Y,XTY_{\text{P}}\}$, $T_1 = C_1 + \lambda$, $T_2 = C_2 + \lambda$, and we get a function $P(x, \lambda)$ for $Y$ evaluated at $x=\overline{Y}(\lambda)$: 
\begin{align}
    P([\overline{Y}](\lambda),\lambda) &= \lambda(\frac{a}{d}[\overline{Y}](\lambda)+\frac{b}{d}-[\overline{Y}](\lambda))\\
    &+\frac{a}{d}(C_2 - c - [\overline{Y}](\lambda))[\overline{Y}](\lambda) + \frac{b}{d}(C_2 - c - [\overline{Y}](\lambda)) - C_1[\overline{Y}](\lambda).
\end{align}
Thus, 
\begin{align}
    q([\overline{Y}]) = (\frac{a}{d}-1)[\overline{Y}]+\frac{b}{d}.
\end{align}
Note that $a>d$, so $q$ does not have any positive roots. Therefore, by Remark~\ref{rmk:root-of-q}, 
\begin{equation}
\lim_{\lambda\to\infty} [\overline{Y}](\lambda)=\infty.
\end{equation}
From this, now we can obtain the limit for $X_j=X_\text{P}$ when $X_i \in \{X_\text{P}Y,XTY_{\text{P}}\}$. Indeed, consider the conservation law given by taking the first conservation law and subtracting the second one. Then we have
\begin{align}
    [\overline{X}] + [\overline{XT}] + [\overline{X_\text{P}}] - [\overline{Y}] - [\overline{Y_\text{P}}] = T_1 - T_2.
\end{align}
Note that this conservation law is not affected by changing the initial concentration of $X_i \in \{X_\text{P}Y,XTY_{\text{P}}\}$, so it will stay fixed as $\lambda\to\infty$. Now, using the steady-state formulas \eqref{eq:biochemex1-ss-for-X}--\eqref{eq:biochemex1-ss-for-XTY_P}, we see that
\begin{align}
    T_1 - T_2 &= [\overline{X}] + [\overline{XT}] + [\overline{X_\text{P}}] - [\overline{Y}] - [\overline{Y_\text{P}}]\notag\\
    &= \frac{(k_{-1}+k_2)k_4}{k_1k_2}[\overline{X_\text{P}Y}] + \frac{k_4}{k_2}[\overline{X_\text{P}Y}] + [\overline{X_\text{P}}] - [\overline{Y}] - \frac{k_2(k_{-5}+k_6)}{k_5k_6}\notag\\
    &= (a-d)[\overline{X_\text{P}Y}] + [\overline{X_\text{P}}] - [\overline{Y}] - c\\
    &= \frac{a-d}{b}[\overline{X_\text{P}}][\overline{Y}] + [\overline{X_\text{P}}] - [\overline{Y}] - c,\notag
\end{align}
where $a,b$ and $d$ are constants depending of $k_i$ as defined before. Then solving for $[\overline{X_\text{P}}]$, we obtain
\begin{align}
    [\overline{X_\text{P}}] = \dfrac{[\overline{Y}] + T_1 - T_2 + c}{\dfrac{a-d}{b}[\overline{Y}] + 1}.
\end{align}
Note that as $\lambda\to\infty$, if $T_1 = C_1+\lambda, T_2 = C_2 + \lambda$, then $T_1 - T_2 = C_1 - C_2$ fixed, and as the limit of $[\overline{Y}](\lambda)$ is $\infty$, we see that
\begin{align}
    \lim_{\lambda\to\infty} [\overline{X_\text{P}}](\lambda) = \dfrac{b}{a-d},
\end{align}
showing that $X_\text{P}$ has aACR when $X_i \in \{X_\text{P}Y,XTY_{\text{P}}\}$.

Putting all these together, we obtain the complete characterization summarized in Table~\ref{tab:SF-2B-unmodified}. See Supplementary Materials for the same analysis for Table~\ref{tab:SF-2B-modified-1}. 
\end{example}


\color{blue}

\color{black}

\section{Discussion}
In this work, we introduced the previously unacknowledged notion of aACR. We showed that approximate concentration robustness is not  merely the outcome of ``exact'' ACR obscured by experimental noise or hidden reactions with negligibly small rate constants. We demonstrated that such approximate robustness can emerge directly from network structure itself. Specifically, whenever a network exhibits aACR for a given species, this approximate invariance arises structurally, without requiring the strong assumption of rate constants being negligible. 
Moreover, we found that aACR is strikingly widespread across biologically relevant systems, including the \textit{E. coli} EnvZ-OmpR osmoregulation pathway and the phosphorylation-dephosphorylation futile cycle
(Fig.~\ref{fig:biochem-example} and Fig.\ref{Supp-fig:additional-networks}; Tables~\ref{tab:SF-2B-unmodified} and \ref{tab:SF-2B-modified-1}; and Tables~\ref{Supp-tab:FutCyc-unmodified} and \ref{Supp-tab:FutCyc-modified-3}). 
We showed that this prevalence is not coincidental by proving the theorems ensuring that aACR arises as a necessary consequence of structural features in networks with conserved quantities (Theorems~\ref{thm:main_theorem} and \ref{thm:second_theorem}). Furthermore, we provide the general approach to completely characterize steady-state responses under perturbation of initial conditions based on computational algebra, as described in \cref{ex:analysis-biochem-unmod}.
%

While ACR has traditionally been the focus of efforts to identify \textit{exact} robustness in reaction networks, our findings suggest that this perspective may be overly restrictive. If one seeks only networks exhibiting exact ACR, many biologically relevant systems that display robust behavior in practice will be overlooked. Networks with aACR can maintain stable concentrations of key species over a wide range of conditions, even if they do not meet the formal criteria for ACR. This highlights the need to broaden our search criteria: by accounting for approximate behaviors that arise intrinsically from network architecture---rather than dismissing them as deviations or noise---we can uncover a wider and more realistic set of candidate networks for understanding or engineering robust biological functions.

Furthermore, our results challenge the common assumption that approximate concentration robustness necessarily arises from a ``core'' ACR subnetwork embedded within a larger system~\cite{ShinarFeinberg2010Science}. Here, we showed that aACR, which is an instance of approximate robustness, can exist even when no identifiable ACR core is present. For example, in Tables~\ref{tab:SF-2B-unmodified} and \ref{tab:SF-2B-modified-1}, aACR frequently appears in species that are irrelevant to any known core ACR structure. Similarly, in the phosphorylation-dephosphorylation futile cycle (Fig.~\ref{Supp-fig:additional-networks}), no instances of ACR are observed, while many species exhibit aACR (Tables~\ref{Supp-tab:FutCyc-unmodified} and \ref{Supp-tab:FutCyc-modified-3}).
That is, the behavior cannot be attributed to a subsystem with exact ACR that is merely perturbed by additional reactions or noise. Instead, the robustness emerges from the interplay of reactions across the full network, often involving mechanisms that fall outside the scope of traditional ACR theory. This observation suggests that aACR represents a genuinely distinct form of robustness, so this deserves separate theoretical treatment and recognition in both modeling and experimental analysis.


The concept of aACR also opens promising avenues for synthetic biology, particularly for the design of biochemical circuits used in computing and control. Recent works by Khammash and colleagues \cite{aoki2019universal, Lillacci-synthetic2018, Anastassov2023}, as well as by Anderson and Joshi \cite{AndersonJoshi2025}, have explored biochemical systems as foundational elements for molecular controllers and logic gates. For such systems to function reliably, intermediate modules must produce consistent outputs, even when connected to downstream components or embedded within larger networks. 
Our results suggest that aACR is especially well-suited for these contexts: the output species in an aACR module remains stable substantial variation in initial conditions. This insensitivity to upstream or downstream perturbations enables robust signal transmission across modular layers, making aACR an attractive structural design principle for biochemical computing. In this way, aACR may provide the robustness backbone needed to build scalable, fault-tolerant synthetic networks.




\appendix

\section*{Supplementary Materials}
This article has accompanying supplementary material PDF file. The file contains supplementary texts for mathematical details and additional examples, supplementary figure, and supplementary tables.

\section*{Acknowledgments}
We thank Bal\'{a}zs Boros for valuable comments and discussions. We also thank Gautam Neelakantan Memana for his valuable help and discussions.

\section{Appendix}

\subsection{Proof of Theorems~\ref{thm:main_theorem} and~\ref{thm:second_theorem}}\label{appendix:proofs}
\begin{proof}[Proof of Theorem~\ref{thm:main_theorem}]
Recall that we say $X_j$ has aACR with respect to $X_i$ if the steady-state value of $X_j$, denoted by $\overline{X}_j(\lambda)$, has a finite and positive limit as $\lambda \to \infty$ (Definition~\ref{def:aACR}). Here, $\lambda$ represents the initial concentration of $X_i$, which we treat as an ``input'' value, and the ``output'' value is the corresponding steady-state value of $X_j$.

For sufficiently large $\lambda$, let $\overline{X}(\lambda) = (\overline{X}_1(\lambda), \dots, \overline{X}_n(\lambda)) \in \mathbb{R}^n_{\geq 0}$ denote the unique positive equilibrium in the stoichiometric compatibility class of $\mathbf{x}_0 + \lambda \mathbf{e}_i$. With this notation, $\overline{X}_j(\lambda)$ represents the dose–response curve of $X_j$ with respect to $X_i$.

We first want to show that the function $\overline{X}_j(\lambda)$ satisfies the properties that are enjoyed by of the function $x(\lambda)$ from the statement of Lemma~\ref{lem:LimitExists}, i.e., it is a continuous algebraic function for a sufficiently large $\lambda$. Note that $\overline{X}(\lambda) = (\overline{X}_1(\lambda), \dots, \overline{X}_n(\lambda))$ is a solution of the system of polynomial (steady-state) equations and conservation relations given by
\begin{align}\label{eq:PolynomialSys}
    f_1(\overline{X}_1(\lambda), \dots, \overline{X}_n(\lambda))&=0\nonumber\\
    \vdots\nonumber\\
    f_n(\overline{X}_1(\lambda), \dots, \overline{X}_n(\lambda))&=0\\
    \sum_{k=1}^n \alpha_k^{(1)} \overline{X}_k(\lambda) &= C^{(1)} + \alpha_i^{(1)}\lambda\nonumber\\
    \vdots\nonumber\\
    \sum_{k=1}^n \alpha_k^{(m)} \overline{X}_k(\lambda) &= C^{(m)} + \alpha_i^{(m)}\lambda\nonumber
\end{align}
where we can assume $f_i:\mathbb{R}^n \to \mathbb{R}$ are polynomials without loss of generality.

Since the function $\overline{X}(\lambda)$ is a semi-algebraic function (i.e., solution of a polynomial system with some sign conditions), and $\overline{X}_j(\lambda)$ is it obtained from it by projection, it follows that $\overline{X}_j(\lambda)$ is a semi-algebraic function by the Tarski–Seidenberg theorem~\cite{bochnak_coste_roy_Real_Alg_Geom}. Then, the graph of $\overline{X}_j(\lambda)$ is a semi-algebraic set and can be written as a finite union of points and curves that admit parameterizations via diffeomorphisms to the interval $(0,1)$ (see for example Section 3 and especially Corollary 3.8 in~\cite{coste_Intr_Semialg_Geometry}). In particular, since the number of such curves is finite, it follows that for $\lambda$ large enough the graph of $\overline{X}_j(\lambda)$ consists of a single such curve. If we denote this curve by $\gamma(s) = (g(s),h(s))$, then we have 
$$
\overline{X}_j(g(s)) = h(s),
$$
for all $s \in (0,1)$. Since $\gamma:(0,1)\to\mathbb{R}^2$ is a diffeomorphism and the graph of $\overline{X}_j(\lambda)$ satisfies the vertical line test, it follows that $g(s)$ must be one-to-one and thus a homeomorphism. Then we obtain 
$$
\overline{X}_j(\lambda) = h(g^{-1}(\lambda)),
$$
and therefore $\overline{X}_j(\lambda)$ is continuous for a sufficiently large $\lambda$. 
The fact that $\overline{X}_j(\lambda)$ is an algebraic function (i.e., there exists a nonzero polynomial $P(x,\lambda)$ such that $P(\overline{X}_j(\lambda),\lambda)=0$) follows similarly from properties of semi-algebraic functions. A short but detailed explanation appears for example in the proof of Proposition 2.11 in~\cite{coste_Intr_Semialg_Geometry}.

We can conclude that the function $\overline{X}_j(\lambda)$ satisfies the hypotheses of Lemma~\ref{lem:LimitExists}, and therefore the limit $$
\displaystyle \lim_{\lambda \to \infty} \overline{X}_j(\lambda)
$$ 
exists; moreover, since the species $X_j$ is contained in the support of a positive conservation law that does not change as $\lambda$ changes, it follows that this limit is finite. 

If we assume that this limit is zero for all $\overline{X}_j \in \Sigma$, then we arrive at a contradiction: it implies that every $X_j$ is identically zero for all time, which is impossible due to the conservation law. Therefore, there must exist at least one species $\overline{X}_j \in \Sigma$ for which the limit $\displaystyle\lim_{\lambda \to \infty} \overline{X}_j(\lambda)$ is finite and positive. This, in turn, implies that there exists at least one $X_j$ that has aACR with respect to $X_i$.
\end{proof}

\subsection{Model-specific derivations and analyses for the examples in Fig.~\ref{fig:archetypal}}\label{subsec:model-derivation}
Consider first the reaction network in~\eqref{crn:archetypal} (Fig.~\ref{fig:archetypal}a). As we derived in \eqref{eq:archetypal-SS-1}, 
the steady-state variety is given by
\begin{equation}
    - \alpha X^{\text{SS}}Y^{\text{SS}} + \beta Y^{\text{SS}} =(- \alpha X^{\text{SS}} + \beta) Y^{\text{SS}} = 0.
\end{equation}
Thus, either $Y^{\text{SS}} = 0$ or $- \alpha X^{\text{SS}} + \beta = 0$. Using that $X^{\text{SS}} + Y^{\text{SS}} = T$ we get
\begin{equation}
    (X^{\text{SS}}, Y^{\text{SS}}) = \begin{cases}
        \left(\dfrac{\beta}{\alpha}, T-\dfrac{\beta}{\alpha}\right) & \text{if}\quad  T \geq \beta/\alpha , \\ 
        \left(T,0\right) & \text{otherwise.}
    \end{cases}
\end{equation}
Now, let us focus on the modified reaction network shown in Fig.~\ref{fig:archetypal}d. The corresponding dynamical system can be described by:
\begin{equation}\label{eq:archetypal-AsymACR}
    \begin{aligned}
        \dfrac{dX}{dt} & = - \alpha XY + \beta Y - \gamma X,\\
        \dfrac{dY}{dt} & =  \alpha XY - \beta Y + \gamma X .\\
    \end{aligned}
\end{equation}

It is still true that we have $T = X(t) + Y(t)$ constant. The steady-state variety now is given by
\begin{equation}
    - \alpha X^{\text{SS}}Y^{\text{SS}} + \beta Y^{\text{SS}} - \gamma X^{\text{SS}} = 0.
\end{equation}
Using $T = X^{\text{SS}} + Y^{\text{SS}}$, we can substitute and calculate the steady states in terms of $T$:
\begin{align}
    X^{\text{SS}} &= \dfrac{1}{2\alpha} (\beta + \gamma + \alpha T - \sqrt{\Delta}),\\
    Y^{\text{SS}} &= \dfrac{1}{2\alpha} (-\beta - \gamma + \alpha T + \sqrt{\Delta}),
\end{align}
where
\begin{align*}
    \Delta &= \beta^2 + \gamma^2 +\alpha^2T^2 + 2\beta\gamma + 2\alpha\gamma T - 2\alpha \beta T = (\alpha T - \beta + \gamma)^2 + 4\beta\gamma.
\end{align*}
Note that this solution is unique with the restriction that both $A, B \geq 0$. We then see that as $T\to \infty$, $\sqrt{\Delta} \to \alpha T - \beta + \gamma$ and so
\begin{equation}
    X^{\text{SS}} \approx \dfrac{1}{2\alpha} (\beta + \gamma + \alpha T - (\alpha T - \beta + \gamma)) = \dfrac{\beta}{\alpha}.
\end{equation}
Note that here we calculated $X^{\text{SS}}$ in terms of $T$. Nevertheless, even if either $X(0)$ or $Y(0)$ is chosen as an input, $T$ also goes to infinity as $\lambda$ goes to infinity. Thus, the limit $\lim_{\lambda \to \infty}X^{\text{SS}}(\lambda)$ is positive and finite, implying that the species $X$ exhibits aACR with respect to either $X$ or $Y$ according to Definition~\ref{def:aACR}.

\begin{remark}
Note that we have an even stronger result in this case. Specifically, if we compute the steady-state variety over all of $\mathbb{R}^2$ and separate the stable and unstable varieties, then as $\gamma \to 0$, the stable variety of \eqref{eq:archetypal-AsymACR} approaches uniformly to the stable variety of \eqref{eq:archetypal-ODE}. The same holds for the respective unstable varieties.
\end{remark}

\bibliographystyle{siamplain}
\bibliography{refs}
\end{document}


\maketitle
In Supplementary Materials, we present various examples of realistic biochemical systems. Specifically, we demonstrate how our mathematical theorems (Theorems~\ref{Main-thm:main_theorem} and~\ref{Main-thm:second_theorem} and Lemma~\ref{Main-lem:LimitExists}) can be used to \textit{completely}  steady-state responses of the system (e.g., Tables~\ref{Main-tab:SF-2B-unmodified} and \ref{Main-tab:SF-2B-modified-1}). 

\section{Applications of Theorem~\ref{Main-thm:main_theorem}}
Theorem~\ref{Main-thm:main_theorem} guarantees the existence of aACR under the presence of positive conservation laws, without the need of a detailed analysis. 
In this section, we demonstrate that how Theorem~\ref{Main-thm:main_theorem} is applied to other examples, such as a phosphorylation-dephosphorylation futile cycle and its variant. 

\subsection{Phosphorylation-dephosphorylation futile cycle (Fig.~\ref{fig:additional-networks})} 

\begin{example}
We now consider a reaction network modeling phosphorylation-dephosphorylation futile cycle in Fig.~\ref{fig:additional-networks}. This network has the following three conservation laws: 
\begin{align}
    [S] + [P] + [SE] + [PF] &= T_1,\\
    [E] + [SE] &= T_2,\\
    [F] + [PF] &= T_3.
\end{align}
Again, note that these do not change under modifications where irreversible reaction are changed to be reversible. By Theorem~\ref{Main-thm:main_theorem}, we see that
\begin{itemize}
    \item For each $X_i \in \{E,F\}$, there is at least one $X_j \in \Sigma_{T_1} \coloneqq \{S,P,SE,PF\}$ that admits aACR with respect to $X_i$.
    \item For each $X_i\in\{S,P,F,PF\}$, there is at least one $X_j \in \Sigma_{T_2} \coloneqq \{E,SE\}$ that admits aACR with respect to $X_i$.
    \item For each $X_i\in\{S,P,E,SE\}$, there is at least one $X_j \in \Sigma_{T_3} \coloneqq\{F,PF\}$ that admits aACR with respect to $X_i$.
\end{itemize}
which explains aACR even for the network modification (See Table~\ref{tab:FutCyc-modified-3}).
\end{example}

\section{Applications of Lemma~\ref{Main-lem:LimitExists} and Theorem~\ref{Main-thm:second_theorem} allowing complete characterization of steady-state behaviors}
We can use Lemma~\ref{Main-lem:LimitExists} and Theorem~\ref{Main-thm:second_theorem} to provide a complete characterization of steady-state behaviors. Even more, in the following examples whenever the limit behavior of a dose-response curve has aACR, we explicitly calculate its limit.

\subsection{Modified network for the \textit{E. coli} EnvZ-OmpR system  (Fig.~\ref{Main-fig:biochem-example}d)}

\begin{example}
Consider the modified network (Fig.~\ref{Main-fig:biochem-example}D in the main text) given by making the first irreversible reaction (i.e., $XT \rightarrow X_\text{P}$) reversible in the reaction network above. The corresponding system of ODEs is given by 
\begin{align*}
    \dfrac{d[X]}{dt} &= -k_1[X] + k_{-1} [XT] + k_3 [X_\text{P}Y], \\ 
    \dfrac{d[XT]}{dt} &= k_1[X] - (k_{-1}+k_{-1})[XT] - k_{-3}[XT][Y_\text{P}] \\ 
    & \phantom{XXXXXXXXXX} + (k_{-5}+k_4)[XTY_\text{P}] + k_{-2}[X_\text{P}], \\
    \dfrac{d[X_\text{P}]}{dt} &= k_{-1}[XT] - k_2[X_\text{P}][Y] + k_{-3}[X_\text{P}Y] - k_{-2}[X_\text{P}], \\ 
    \dfrac{d[Y]}{dt} &= k_4[XTY_\text{P}] - k_2[X_\text{P}][Y] + k_{-3}[X_\text{P}Y], \\
    \dfrac{d[Y_\text{P}]}{dt} &= k_3[X_\text{P}Y] - k_{-3}[XT][Y_\text{P}] + k_{-5}[XTY_\text{P}],\\ 
    \dfrac{d[X_\text{P}Y]}{dt} &= k_2[X_\text{P}][Y] - (k_{-3}+k_3)[X_\text{P}Y], \\ 
    \dfrac{d[XTY_\text{P}]}{dt} &= k_{-3}[XT][Y_\text{P}] - (k_{-5}+k_4)[XTY_\text{P}].
\end{align*}
This system has two conserved quantities:
\begin{itemize}
    \item The total mass of $X$-related species: $[X] + [XT] + [X_\text{P}] + [X_\text{P}Y] + [XTY_\text{P}]= T_1.$
    \item The total mass of $Y$-related species: $[Y] + [Y_\text{P}] + [X_\text{P}Y] + [XTY_\text{P}]=T_2.$
\end{itemize}

As we saw in the previous subsection, just from this information, we can apply Theorem~\ref{Main-thm:main_theorem} to conclude the following:
\begin{itemize}
    \item At least one of the $X$-related species: $X, XT, X_\text{P}, X_\text{P}Y,$ or $XTY_\text{P}$ has aACR with respect to $Y$ and similarly for $Y_\text{P}$.
    \item At least one of the $Y$-related species: $Y, Y_\text{P}, X_\text{P}Y,$ or $XTY_\text{P}$ has aACR with respect to $X$ and similarly for $XT$ and $X_\text{P}$.
\end{itemize}
Moreover, in this case, we can go further and identify exactly which species exhibit aACR through a more detailed analysis.

First, note that through algebraic manipulation we get
\begin{align}
    [\overline{X}] &= \frac{1}{k_1}(k_{-1}[\overline{XT}] + k_6[\overline{XTY_\text{P}}]), \label{eq:biochem1-ss-for-X}\\ 
    [\overline{X_\text{P}}] &= \frac{1}{k_{-2}}(k_{2}[\overline{XT}] - k_6[\overline{XTY_\text{P}}]), \label{eq:biochem1-ss-for-Xp}\\ 
    [\overline{X_\text{P}Y}] &= \frac{k_6}{k_4}[\overline{XTY_\text{P}}], \label{eq:biochem1-ss-for-XpY}\\ 
    [\overline{Y_\text{P}}] &= \frac{k_{-5}+k_6}{k_5} \frac{[\overline{XTY_\text{P}}]}{[\overline{XT}]},\label{eq:biochem1-ss-for-Yp} \\ 
    [\overline{Y}] &= \frac{k_{-3}+k_4}{k_3} \frac{[\overline{X_\text{P}Y}]}{[\overline{X_\text{P}}]}.\label{eq:biochem1-ss-for-Y}
\end{align}

Thus, we can calculate the steady-state variety formula in terms of $[\overline{XTY_\text{P}}]$ and $[\overline{XT}]$ by noting that: 
\begin{align}
    [\overline{Y}] &= \frac{(k_{-3}+k_4)k_6k_{-2}}{k_3k_4} \cdot \frac{[\overline{XTY_\text{P}}]}{k_2[\overline{XT}] - k_6[\overline{XTY_\text{P}}]}.\label{eq:biochem1-ss-Y-with-XT-XTYp}
\end{align}

While this formula is the function of $[\overline{XTY_\text{P}}]$ and $[\overline{XT}]$ in addition to rate constants, we can obtain the formula in terms of $T_1$ and $T_2$, which is more natural. Specifically, first consider the first conservation law:
\begin{align*}
    T_1 & = [\overline{X}] + [\overline{XT}] + [\overline{X_\text{P}}] + [\overline{X_\text{P}Y}] + [\overline{XTY_\text{P}}] \\
    & = \frac{1}{k_1}(k_{-1}[\overline{XT}] + k_6[\overline{XTY_\text{P}}]) + [\overline{XT}] + \frac{1}{k_{-2}}(k_{2}[\overline{XT}] - k_6[\overline{XTY_\text{P}}]) \\
    & \phantom{XXXXXXXXXXXXXXXXXXXXXXXX} + \frac{k_6}{k_4}[\overline{XTY_\text{P}}] + [\overline{XTY_\text{P}}] \\
    &=\left(\frac{k_{-1}}{k_1} + \frac{k_2}{k_{-2}} + 1\right)[\overline{XT}] + \left(\frac{1}{k_{1}}-\frac{1}{k_{-2}}+\frac{1}{k_4}+\frac{1}{k_6}\right)k_6[\overline{XTY_\text{P}}].
\end{align*}

Thus, we can rewrite this formula as
\begin{align}\label{eq:biochem1-T1-ss-for-XT}
    [\overline{XT}] &= a_1 T_1 + b_1 [\overline{XTY_\text{P}}],
\end{align}
where $a_1$ and $b_1$ are constants in terms of $k_i$.

Then, substituting this expression in \eqref{eq:biochem1-ss-for-X}--\eqref{eq:biochem1-ss-for-Y}, to obtain a formula for the steady state in terms of $[\overline{XTY_\text{P}}]$:
\begin{align}
    [\overline{X_\text{P}}] &= a_2 T_1 + b_2 [\overline{XTY_\text{P}}], \label{eq:biochem1-T1-ss-for-Xp}\\ 
    [\overline{X}] &= a_3 T_1 + b_3 [\overline{XTY_\text{P}}], \label{eq:biochem1-T1-ss-for-X}\\ 
    [\overline{X_\text{P}Y}] &= b_4[\overline{XTY_\text{P}}], \label{eq:biochem1-T1-ss-for-XpY}\\ 
    [\overline{Y_\text{P}}] &=\frac{c_1[\overline{XTY_\text{P}}]}{a_1 T_1 + b_1 [\overline{XTY_\text{P}}]}, \label{eq:biochem1-T1-ss-for-Yp}\\ 
    [\overline{Y}] &= \frac{c_2[\overline{XTY_\text{P}}]}{a_2 T_1 + b_2 [\overline{XTY_\text{P}}]},\label{eq:biochem1-T1-ss-for-Y}
\end{align}
where $a_i,b_i,c_i$ are all constants in terms of $k_i$.

From this, we can reduce into a single polynomial equation in terms of $[\overline{XTY_\text{P}}]$ and $T_1,T_2$ to then calculate the limit when $XTY_\text{P}$ is chosen to be $X_j$. Indeed, from the second conservation law,
\begin{align}
    T_2 & = [\overline{Y}] + [\overline{Y_\text{P}}] + [\overline{X_\text{P}Y}] + [\overline{XTY_\text{P}}]\nonumber \\
    &= \frac{c_1[\overline{XTY_\text{P}}]}{a_1 T_1 + b_1 [\overline{XTY_\text{P}}]} + \frac{c_2[\overline{XTY_\text{P}}]}{a_2 T_1 + b_2 [\overline{XTY_\text{P}}]} + (1 + b_4)[\overline{XTY_\text{P}}],
\end{align}
and after clearing denominators we get
\begin{align}
     &((1 + b_4)[\overline{XTY_\text{P}}] - T_2)(a_1 T_1 + b_1 [\overline{XTY_\text{P}}])(a_2 T_1 + b_2 [\overline{XTY_\text{P}}])\nonumber\\ 
     &+ c_1[\overline{XTY_\text{P}}](a_2 T_1 + b_2 [\overline{XTY_\text{P}}]) + c_2[\overline{XTY_\text{P}}](a_1 T_1 + b_1 [\overline{XTY_\text{P}}])= 0.
\end{align}
Therefore, if $X_i \in \{X,XT,X_{\text{P}}\}$, $T_1 = C_1 + \lambda$, and we get a function $P(x, \lambda)$ for $XTY_\text{P}$ evaluated at $x=[\overline{XTY_\text{P}}](\lambda)$: 
\begin{align}
    P([\overline{XTY_\text{P}}](\lambda),\lambda) &= \lambda^2a_1a_2((1 + b_4)[\overline{XTY_\text{P}}](\lambda) - T_2) + \textit{linear term} + \textit{constant term}. 
\end{align}
Thus, 
\begin{align}
    q([\overline{XTY_\text{P}}](\lambda)) = a_1a_2((1 + b_4)[\overline{XTY_\text{P}}](\lambda) - T_2).
\end{align}
Applying Lemma~\ref{Main-lem:LimitExists} and by Remark~\ref{Main-rmk:root-of-q}, we see that the limit of $[\overline{XTY_\text{P}}](\lambda)$ must be a root of $q$, so
\begin{equation}
\lim_{\lambda\to\infty} [\overline{XTY_\text{P}}](\lambda)=\dfrac{T_2}{ 1+ b_4}.
\end{equation}
By \eqref{eq:biochem1-T1-ss-for-XT}--\eqref{eq:biochem1-T1-ss-for-Y}, we see that $XTY_\text{P}$ and $X_\text{P}Y$ have aACR; $X, XT, X_\text{P}$ diverge; and $Y$ and $ Y_\text{P}$ go to 0 with respect to $X_i \in \{X,XT,X_{\text{P}}\}$.

Similarly, if $X_i \in \{Y,Y_{\text{P}}\}$,
\begin{align}
     q([\overline{XTY_\text{P}}](\lambda)) = -(a_1 T_1 + b_1 [\overline{XTY_\text{P}}])(a_2 T_1 + b_2 [\overline{XTY_\text{P}}])
\end{align}
Thus, 
\begin{equation}
\lim_{\lambda\to\infty} [\overline{XTY_\text{P}}](\lambda)=-\dfrac{a_1T_1 }{b_1} \text{ or } -\dfrac{a_2T_1}{b_2}.
\end{equation}
Thus, $XTY_\text{P}$ has aACR with respect to $X_i \in \{Y,Y_{\text{P}}\}$. 
By \eqref{eq:biochem1-T1-ss-for-XT}--\eqref{eq:biochem1-T1-ss-for-Y}, $X_\text{P}Y$ also has aACR with respect to $X_i \in \{Y,Y_{\text{P}}\}$. We also see that if the limit is the first quantity, then $XT$ will go to 0, while if it is the second quantity, $X_\text{P}$ will go to 0. 

But, notice that we can rewrite~\eqref{eq:biochem1-ss-for-Xp} as
\begin{align}\label{eq:biochem1-ss-for-XT-with-Xp-XTYp}
    [\overline{XT}] = \frac{1}{k_2}(k_{-2}[\overline{X_\text{P}}]+k_6[\overline{XTY_\text{P}}]),
\end{align}
which shows that the limit of $[\overline{XT}](\lambda)$ cannot be 0. Thus,
\begin{align}
    \lim_{\lambda\to\infty} [\overline{XTY_\text{P}}](\lambda)&=-\dfrac{a_2T_1}{b_2},\\
    \lim_{\lambda\to\infty} [\overline{X_\text{P}}](\lambda)&=0.
\end{align}
Then by \eqref{eq:biochem1-ss-for-XT-with-Xp-XTYp} and \eqref{eq:biochem1-ss-for-Xp} implies
\begin{align*}
    \lim_{\lambda\to\infty} [\overline{XT}](\lambda)&= \frac{k_{-2}}{k_2} \times \lim_{\lambda\to\infty}[\overline{X_\text{P}}](\lambda) + \frac{k_6}{k_2} \times\lim_{\lambda\to\infty}[\overline{XTY_\text{P}}](\lambda) 
    \\  & = -\dfrac{k_6a_2T_1}{k_2b_2}, \\
    \lim_{\lambda\to\infty} [\overline{X}](\lambda)&= \frac{1}{k_1}(k_{-1}\times\lim_{\lambda\to\infty}[\overline{XT}](\lambda) +k_6 \times \lim_{\lambda\to\infty}[\overline{XTY_\text{P}}](\lambda)) \\ 
    &= -\frac{1}{k_1}(\frac{k_{-1}k_6}{k_2}+k_6)\dfrac{a_2T_1}{b_2},
\end{align*}
which shows that both $X$ and $XT$ have aACR, as well as $Y_p$, with respect to $X_i \in \{Y,Y_{\text{P}}\}$. Also, $Y$ diverges with respect to $X_i \in \{Y,Y_{\text{P}}\}$. In particular, by~\eqref{eq:biochem1-ss-for-Yp}, we have
\begin{align}
     \lim_{\lambda\to\infty} [\overline{Y_\text{P}}](\lambda)&= \frac{k_{-5}+k_6}{k_5} \times\lim_{\lambda\to\infty}\frac{[\overline{XTY_\text{P}}]}{[\overline{XT}]} = \frac{k_2(k_{-5}+k_6)}{k_5 k_6},
\end{align}
which is \emph{precisely the ACR value of species $Y_\text{P}$ for the unmodified network (see Example~\ref{Main-ex:analysis-biochem-unmod}).}

Finally, if $X_i \in \{X_\text{P}Y,XTY_{\text{P}}\}$, $T_1 = C_1 + \lambda, T_2 = C_2 + \lambda$, and we get a function $P(x, \lambda)$ for $XTY_\text{P}$ evaluated at $x=[\overline{XTY_\text{P}}](\lambda)$: 
\begin{align}
    P([\overline{X_\text{P}Y}](\lambda),\lambda) &= -\lambda^3a_1a_2 +\dots.
\end{align}
Thus,
\begin{equation}
    q([\overline{XTY_\text{P}}](\lambda))=-a_1a_2.
\end{equation}
Therefore, as $q$ does not have any positive roots, 
\begin{equation}
\lim_{\lambda\to\infty} [\overline{XTY_\text{P}}](\lambda)=\infty, 
\end{equation}
by Remark~\ref{Main-rmk:root-of-q}. From \eqref{eq:biochem1-T1-ss-for-XpY} we still see that $X_\text{P}Y$ also diverges. To see the behavior of $X$ and $XT$, we note the following identities:
\begin{align}
     k_{-1}[\overline{XT}] + k_6[\overline{XTY_\text{P}}] &= k_1[\overline{X}],\\
    k_1[\overline{X}] + k_{-2}[\overline{X_\text{P}}] &= (k_{-1} + k_2)[\overline{XT}],
\end{align}
where the first expression shows $X$ also diverges, and the second that $XT$ diverges. Now we will analyze the behavior of $Y, Y_\text{P}$ and $X_\text{P}$ with respect to $X_i \in \{X_\text{P}Y,XTY_{\text{P}}\}$. 

Let $X_j=Y$. Consider the conservation law given by taking the first conservation law and subtracting the second one. That is,
\begin{align}\label{eq:biochem-conservation3}
    [\overline{X}] + [\overline{XT}] + [\overline{X_\text{P}}] - [\overline{Y}] - [\overline{Y_\text{P}}] = T_1 - T_2.
\end{align}
Note that this conservation law is not affected by changing the initial concentration of $X_i \in \{X_\text{P}Y,XTY_{\text{P}}\}$, so it will stay fixed as $\lambda\to\infty$. Now, using the steady-state formulas \eqref{eq:biochem1-ss-for-X}--\eqref{eq:biochem1-ss-for-Y}, we see that
\begin{align}
    T_1 - T_2 &= [\overline{X}] + [\overline{XT}] + [\overline{X_\text{P}}] - [\overline{Y}] - [\overline{Y_\text{P}}]\notag\\
    &= \frac{k_{-1}}{k_1}[\overline{XT}] + \frac{k_6}{k_1}[\overline{XTY_\text{P}}] + [\overline{XT}] + \frac{k_2}{k_{-2}}[\overline{XT}] - \frac{k_6}{k_{-2}}[\overline{XTY_\text{P}}] - [\overline{Y}] - [\overline{Y_\text{P}}]\notag\\
    &= \left(\frac{k_{-1}}{k_1}+\frac{k_2}{k_{-2}}+1\right)[\overline{XT}] + \left(\frac{1}{k_1}-\frac{1}{k_{-2}}\right)k_6[\overline{XTY_\text{P}}] - [\overline{Y}] - [\overline{Y_\text{P}}]\\
    &= \left(\frac{k_{-1}}{k_1}+\frac{k_2}{k_{-2}}+1\right)[\overline{XT}] + \left(\frac{1}{k_1}-\frac{1}{k_{-2}}\right)\frac{k_5k_6}{k_{-5}+k_6}[\overline{XT}][\overline{Y_\text{P}}] - [\overline{Y}] - [\overline{Y_\text{P}}]\notag\\
    &= A[\overline{XT}] + B[\overline{XT}][\overline{Y_\text{P}}] - [\overline{Y}] - [\overline{Y_\text{P}}],\notag
\end{align}
where $A$ and $B$ are constants depending of $k_i$.
We can write then $[\overline{Y}]$ in terms of $T_1 - T_2$, $[\overline{XT}]$ and $[\overline{Y_\text{P}}]$ and bound $[\overline{Y}]$ as follows:
\begin{align}
    [\overline{Y}](\lambda) &= A[\overline{XT}](\lambda) + (B[\overline{XT}](\lambda) - 1)[\overline{Y_\text{P}}](\lambda) - T_1 + T_2\notag\\
    &\geq A[\overline{XT}](\lambda) - T_1 + T_2,
\end{align}
for sufficiently large $\lambda$. Indeed, since $[\overline{XT}](\lambda)$ diverges, there exists $\lambda_0$ such that $B[\overline{XT}](\lambda) - 1\geq 0$ for all $\lambda > \lambda_0$. Therefore,
\begin{align}
    \lim_{\lambda\to\infty}[\overline{Y}](\lambda) = \infty.
\end{align}

From this, we can calculate the limit for $X_j=Y_\text{P}$. Indeed, note that by \eqref{eq:biochem1-ss-for-Yp} and \eqref{eq:biochem1-ss-Y-with-XT-XTYp}, we have
\begin{align}
    [\overline{Y}] &= \frac{(k_{-3}+k_4)k_6k_{-2}}{k_3k_4} \frac{[\overline{XTY_\text{P}}]}{k_2[\overline{XT}] - k_6[\overline{XTY_\text{P}}]}\notag\\
    &= c_2k_{-2}\dfrac{c_1^{-1}[\overline{XT}][\overline{Y_\text{P}}]}{k_2[\overline{XT}] - k_6c_1^{-1}[\overline{XT}][\overline{Y_\text{P}}]}\label{eq:biochem1-ss-Y-with-Yp}\\
    &=\dfrac{c_2k_{-2}[\overline{Y_\text{P}}]}{c_1k_2 - k_6[\overline{Y_\text{P}}]},\notag
\end{align}
where $c_1$ and $c_2$ are constants in terms of $k_i$ as described before. Then, as $[\overline{Y}]$ diverges in the limit, it must follow that
\begin{align}
    \lim_{\lambda\to\infty} [\overline{Y_\text{P}}] = \frac{c_1k_2}{k_6} = \frac{k_2(k_{-5}+k_6)}{k_5k_6},
\end{align}
which is the ACR value of species $Y_\text{P}$ for the unmodified network as we discussed previously. Thus, $X_j=Y_\text{P}$ has aACR with respect to $X_i \in \{X_\text{P}Y,XTY_{\text{P}}\}$. Note that \eqref{eq:biochem1-ss-Y-with-Yp} shows that
\begin{align}
    [\overline{Y_\text{P}}] \leq \dfrac{c_1k_2}{k_6},
\end{align}
in all cases, which explains why $Y_\text{P}$ either has aACR or approaches 0, but never diverges, independent of the choice for $X_i$.

Finally, we use the same conservation law \eqref{eq:biochem-conservation3} to calculate the limit behavior for $X_j = X_\text{P}$ when $X_i \in \{X_\text{P}Y,XTY_{\text{P}}\}$. First, note that we can write $[\overline{X}]$ and $[\overline{XT}]$ in terms of $[X_\text{P}]$ and $[\overline{XTY_\text{P}}]$ using \eqref{eq:biochem1-ss-for-X} and \eqref{eq:biochem1-ss-for-Xp}:
\begin{align}
    [\overline{XT}] &= \frac{1}{k_2}(k_{-2}[\overline{X_\text{P}}]+k_6[\overline{XTY_\text{P}}]),\\
    [\overline{X}] &= \frac{k_{-1}k_{-2}}{k_1k_2}[\overline{X_\text{P}}] + \frac{(k_{-1}+k_2)k_6}{k_1k_2}[\overline{XTY_\text{P}}],
\end{align}
and we can also write $[\overline{XTY_\text{P}}]$ in terms of $[\overline{X_\text{P}}]$ and $[\overline{Y}]$ using \eqref{eq:biochem1-ss-for-XpY} and \eqref{eq:biochem1-ss-for-Y}:
\begin{align}
    [\overline{XTY_\text{P}}] &= \frac{k_4k_3}{k_6(k_{-3}+k_4)}[\overline{X_\text{P}}][\overline{Y}].
\end{align}
Thus, substituting it back into \eqref{eq:biochem-conservation3} we see
\begin{align}
    T_1 - T_2 &= [\overline{X}] + [\overline{XT}] + [\overline{X_\text{P}}] - [\overline{Y}] - [\overline{Y_\text{P}}]\notag\\
    &=\frac{k_{-1}k_{-2}}{k_1k_2}[\overline{X_\text{P}}] + \frac{(k_{-1}+k_2)k_6}{k_1k_2}[\overline{XTY_\text{P}}] +\frac{k_{-2}}{k_2}[\overline{X_\text{P}}]+\frac{k_6}{k_2}[\overline{XTY_\text{P}}] \\
    & \phantom{XXXXXXXXXXXXXXXXXXXXXXXXXXX} +  [\overline{X_\text{P}}] - [\overline{Y}] - [\overline{Y_\text{P}}]\notag\\
    &= \frac{k_{-1}k_{-2}+k_1k_{-2}+k_1k_2}{k_1k_2}[\overline{X_\text{P}}] + \frac{(k_1 + k_{-1}+k_2)k_6}{k_1k_2}[\overline{XTY_\text{P}}] - [\overline{Y}] - [\overline{Y_\text{P}}]\\
    &=\frac{k_{-1}k_{-2}+k_1k_{-2}+k_1k_2}{k_1k_2}[\overline{X_\text{P}}] + \frac{(k_1 + k_{-1}+k_2)}{k_1k_2}\frac{k_4k_3}{(k_{-3}+k_4)}[\overline{X_\text{P}}][\overline{Y}]  \\
    & \phantom{XXXXXXXXXXXXXXXXXXXXXXXXXXXXXXx} - [\overline{Y}] - [\overline{Y_\text{P}}]\notag\\
    &= C[\overline{X_\text{P}}] + D[\overline{X_\text{P}}][\overline{Y}]- [\overline{Y}] - [\overline{Y_\text{P}}],\notag
\end{align}
where $C$ and $D$ are constants in terms of $k_i$. Then solving for $[\overline{X_\text{P}}]$, we obtain
\begin{align}
    [\overline{X_\text{P}}] = \dfrac{[\overline{Y}]+ T_1 - T_2 + [\overline{Y_\text{P}}]}{D[\overline{Y}] + C}.
\end{align}
Note that as $\lambda\to\infty$, if $T_1 = C_1+\lambda, T_2 = C_2 + \lambda$, then $T_1 - T_2 = C_1 - C_2$ fixed, and as the limit of $[\overline{Y}](\lambda)$ is $\infty$ and the limit of $[\overline{Y_\text{P}}](\lambda)$ is finite, we see that
\begin{align}
    \lim_{\lambda\to\infty} [\overline{X_\text{P}Y}](\lambda) = \dfrac{1}{D},
\end{align}
showing that $X_\text{P}Y$ has aACR when $X_i \in \{X_\text{P}Y,XTY_{\text{P}}\}$.

Putting all these together, we obtain the complete characterization summarized in Table~\ref{Main-tab:SF-2B-modified-1} in our main text.
\end{example}

\subsubsection{Phosphorylation-dephosphorylation futile cycle (Fig.~\ref{fig:additional-networks})} 
\begin{example}
Consider the futile cycle (Fig.~\ref{fig:additional-networks}). It has the following system of ODEs:
\begin{align}
    \dfrac{d[S]}{dt} &= -k_1 [S][E] + k_{-1}[SE] + k_4[PF],\\
    \dfrac{d[P]}{dt} &= -k_3 [P][F] + k_{-3}[PF] + k_2[SE],\\
    \dfrac{d[E]}{dt} &= -k_1 [S][E] + (k_{-1}+k_2)[SE],\\
    \dfrac{d[F]}{dt} &= -k_3 [P][F] + (k_{-3}+k_4)[PF],\\
    \dfrac{d[SE]}{dt} &= k_1 [S][E] - (k_{-1}+k_2)[SE],\\
    \dfrac{d[PF]}{dt} &= k_3 [P][F] - (k_{-3}+k_4)[PF],
\end{align}
with the following conservation laws:
\begin{align}
    [S] + [P] + [SE] + [PF] &= T_1,\\
    [E] + [SE] &= T_2,\label{eq:FutileCycle-cons2}\\
    [F] + [PF] &= T_3.
\end{align}

We note that this system has a unique steady state in each stoichiometric compatibility class by the Deficiency One theorem \cite{feinberg-1987}, so the dose-response curves will always be well-defined. Thus, we can apply Theorem~\ref{Main-thm:main_theorem} and conclude the following:
\begin{itemize}
    \item For each $X_i\in\{E,F\}$, there is at least one $X_j\in\{S,P,SE,PF\}$ that admits aACR with respect to $X_i$.
    \item For each $X_i\in\{S,P,F,PF\}$, there is at least one $X_j\in\{E,SE\}$ that admits aACR with respect to $X_i$.
    \item For each $X_i\in\{S,P,E,SE\}$, there is at least one $X_j\in\{F,PF\}$ that admits aACR with respect to $X_i$.
\end{itemize}
Moreover, in this case, we can go further and identify exactly which species exhibit aACR through a more detailed analysis.

Let us analyze the case when $X_j=[SE]$. 
For this, we want to reduce the steady-state equations with the conservation laws into a unique polynomial equation only in terms of $[\overline{SE}]$ and the conservation quantities $T_1, T_2, T_3$.

Note that at the steady state:
\begin{align}
    [\overline{SE}] &= \dfrac{k_1}{k_{-1}+k_2}[\overline{S}][\overline{E}] = a [\overline{S}](T_2 - [\overline{SE}]),\\
    [\overline{PF}] &= \dfrac{k_3}{k_{-3}+k_4}[\overline{P}][\overline{F}] = b [\overline{P}](T_3 - [\overline{PF}]),\label{eq:ss-PF-P-F}\\
    [\overline{PF}] &= \dfrac{k_2}{k_4} [\overline{SE}] = c[\overline{SE}],
\end{align}
where $a = \frac{k_1}{k_{-1}+k_2}, b = \frac{k_3}{k_{-3}+k_4}$ and $c = \frac{k_2}{k_4}$. Thus, we can express the steady state concentration of all species in terms of $[\overline{SE}]$:
\begin{align}
    [\overline{S}] &= \dfrac{[\overline{SE}]}{a(T_2 - [\overline{SE}])},\label{eq:ss-for-S}\\
    [\overline{P}] &= \dfrac{c[\overline{SE}]}{b(T_3 - c[\overline{SE}])},\label{eq:ss-for-P} \\
    [\overline{E}] &= T_2 - [\overline{SE}],\\
    [\overline{F}] &= T_3 - c[\overline{SE}],\label{eq:ss-for-F}\\
    [\overline{PF}] &= c[\overline{SE}].\label{eq:ss-for-PF}
\end{align}
Note that for these calculations we used the steady state equations and the second and third conservation laws, but not the first one. Thus, substituting it in the first conservation law, we get the reduced polynomial equation:
\begin{equation}
    \dfrac{[\overline{SE}]}{a(T_2 - [\overline{SE}])} + \dfrac{c[\overline{SE}]}{b(T_3 - c[\overline{SE}])} + (1 + c)[\overline{SE}] = T_1,
\end{equation}
which reduces to
\begin{align}\label{eq:FutileCycle-P(SE)}
    &\frac{1}{a}[\overline{SE}](T_3 - c[\overline{SE}]) + \frac{c}{b}[\overline{SE}](T_2 - [\overline{SE}])\\
    &+(1+c)[\overline{SE}](T_2 - [\overline{SE}])(T_3 - c[\overline{SE}]) - T_1(T_2 - [\overline{SE}])(T_3 - c[\overline{SE}]) =0.\nonumber
\end{align}
Therefore, when $X_i \in \{S,P\}$, $T_1 = C_1 + \lambda$, and we obtain a function $P(x, \lambda)$ for $SE$ evaluated at $x=\overline{SE}(\lambda)$:
\begin{align*}
    P([\overline{SE}](\lambda),\lambda) &=\frac{1}{a}[\overline{SE}](\lambda)(T_3 - c[\overline{SE}](\lambda)) + \frac{c}{b}[\overline{SE}](\lambda)(T_2 - [\overline{SE}](\lambda)) \\
    &+ (1+c)[\overline{SE}](\lambda)(T_2 - [\overline{SE}](\lambda))(T_3 - c[\overline{SE}](\lambda)) \\
    & - C_1(T_2 - [\overline{SE}](\lambda))(T_3 - c[\overline{SE}](\lambda))\nonumber -\lambda(T_2 - [\overline{SE}](\lambda))(T_3 - c[\overline{SE}](\lambda)).
\end{align*}
From this, we can obtain the function $q$ as follows:
\begin{equation}
    q([\overline{SE}](\lambda)) = (T_2 - [\overline{SE}](\lambda))(T_3 - c[\overline{SE}](\lambda)).
\end{equation}
Thus, applying Lemma~\ref{Main-lem:LimitExists},
\begin{equation}
    \lim_{\lambda\to\infty} [\overline{SE}](\lambda) = T_2 \text{ or } \frac{T_3}{c}.
\end{equation}
From \eqref{eq:ss-for-S} and \eqref{eq:ss-for-P}, we have natural constraints: 
\begin{align}
    [\overline{S}] &= \dfrac{[\overline{SE}]}{a(T_2 - [\overline{SE}])} \geq0, \\\
    [\overline{P}] & = \dfrac{c[\overline{SE}]}{b(T_3 - c[\overline{SE}])}\geq0.
\end{align}
By complying these constraints, the limit of $[\overline{SE}]$ is given as follows:
\begin{equation}
    \lim_{\lambda\to\infty} [\overline{SE}](\lambda) = \min \{T_2, \frac{T_3}{c}\} \in (0,\infty),
\end{equation}
and so the species $SE$ exhibits aACR with respect to $S$ or $P$. 

So far, we have analyzed the behavior of the species $SE$. We can still use \eqref{eq:ss-for-S}-\eqref{eq:ss-for-PF} to characterize \emph{all} possible behaviors of the remaining species using the behavior of $SE$ in the following three distinct cases:
\begin{itemize}
    \item[(\textit{i})] If $T_2 < \frac{T_3}{c}$, then
    \begin{align*}
        \lim_{\lambda\to\infty}[\overline{S}](\lambda) &= \infty, \\
        \lim_{\lambda\to\infty}[\overline{P}](\lambda) &= \dfrac{cT_2}{b(T_3 - cT_2)} \in (0,\infty), \\
        \lim_{\lambda\to\infty}[\overline{E}](\lambda) &= 0, \\
        \lim_{\lambda\to\infty}[\overline{F}](\lambda) &= T_3 - cT_2 \in (0,\infty), \\
        \lim_{\lambda\to\infty}[\overline{PF}](\lambda) &= cT_2 \in (0,\infty).
    \end{align*}
    In this case, $[P]$, $[F]$, and $[PF]$ admit aACR, while $[S]$ diverges and $[E]$ converges to zero.

    \item[(\textit{ii})] If $T_2 > \frac{T_3}{c}$, then $[S]$, $[E]$, and $[PF]$ admit aACR, while $[P]$ diverges and $[F]$ converges to zero in the limit.

    \item[(\textit{iii})] If $T_2 = \frac{T_3}{c}$, then $[S]$ and $[P]$ both diverge, $[E]$ and $[F]$ converge to zero, and $[PF]$ still admits aACR.
\end{itemize}
This characterizes the behavior for of each species with respect to $X_i\in\{S, P\}$. By performing a similar analysis for other choices of $X_i$, we can use the same procedure from \eqref{eq:FutileCycle-P(SE)} to obtain the complete characterization of Table~\ref{tab:FutCyc-unmodified}. 

Suppose now that $X_i = E$. Then, $T_2 = C_2 + \lambda$, and we get
\begin{equation*}
    q([\overline{SE}]) = \frac{c}{b} + ((1+c)[\overline{SE}] - T_1)(T_3 - C[\overline{SE}]).
\end{equation*}
Thus, since $q([\overline{SE}])$ does not have a root at 0 and $X_i=E$ is not in the support of the first conservation law, which is a positive conservation law containing $SE$, $SE$ has aACR with respect to $X_i=E$ by Theorem~\ref{Main-thm:second_theorem}. By \eqref{eq:ss-for-S}--\eqref{eq:ss-for-PF}, $S$ goes to 0 and $E$ diverges with respect to $X_i=E$. 

We thus see that we do not need to solve this quadratic polynomial to conclude that $SE$ has aACR thanks to the second theorem. But, we do not gain any information about the behavior of $P$ and $F$ with respect to $E$ with this approach. For this we will need to recalculate the polynomial for a different choice of $X_j$.

Let $X_j = F$. Then from \eqref{eq:ss-for-F}, we have
\begin{align*}
    [\overline{SE}] &= \frac{T_3 - [\overline{F}]}{c},
\end{align*}
and after substituting it in \eqref{eq:FutileCycle-P(SE)}, we get
\begin{align*}
    &\frac{1}{a}\frac{T_3 - [\overline{F}]}{c}[\overline{F}] + \frac{1}{b}(T_3 - [\overline{F}])(T_2 - \frac{T_3 - [\overline{F}]}{c})\\
    &+ (\frac{1+c}{c}(T_3 - [\overline{F}]) - T_1)(T_2 - \frac{T_3 - [\overline{F}]}{c})[F] = 0.
\end{align*}
If $X_i = E$, then $T_2 = C_2 + \lambda$ and we obtain the polynomial $q$ as follows:
\begin{align*}
    q([\overline{F}]) = \frac{1}{b}(T_3 - [\overline{F}]) + (\frac{1+c}{c}(T_3 - [\overline{F}]) - T_1)[\overline{F}].
\end{align*}
Thus, since $q([\overline{F}])$ does not have a root at 0 and $X_i=E$ is not in the support of the third conservation law, which is a positive conservation law containing $F$, $F$ has aACR with respect to $X_i=E$ by Theorem~\ref{Main-thm:second_theorem}. From \eqref{eq:ss-PF-P-F} $P$ also has aACR with respect to $X_i=E$ since $PF$ has aACR with respect to $E$.

By a similar calculation as the one we just made, we can see that for $X_i=F$, $SE,PF,E$ and $S$ all have aACR with respect to $F$, while $P$ goes to 0 and $F$ diverges with respect to $X_i=F$.

Now, let $X_i=SE$. Then $T_1 = C_1 + \lambda$, $T_2 = C_2 + \lambda$. Now, we see that the polynomial $q$ is given by:
\begin{align*}
    q([\overline{SE}])=-(T_3 - c[\overline{SE}]),
\end{align*}
which has a positive root $\frac{T_3}{c}$. But, note that directly from the conservation laws we cannot assure that $X_j=SE$ will have a finite limit, so we do not yet conclude if $SE$ has aACR. On the other hand, because of \eqref{eq:ss-for-PF}, we see that $SE$ has a finite limit if and only if, $PF$ has a finite limit. And as $X_i=E$ is not in the support of the third conservation law, which is a positive conservation law containing $PF$, we see that $PF$ will have a finite limit, and thus $SE$ must too. Thus,
\begin{align*}
    \lim_{\lambda\to\infty}[\overline{SE}](\lambda) = \frac{T_3}{c}.
\end{align*}
Consequently, $SE$ has aACR with respect to $X_i=SE$. From \eqref{eq:ss-for-S}--\eqref{eq:ss-for-PF}, $PF$ also has aACR; $F$ and $S$ go to 0; and $P$ and $E$ diverge with respect to $X_i=SE$. We can do an analogous analysis for the case $X_i=PF$.

Putting all these together, these analyses provide the complete characterization of Table~\ref{tab:FutCyc-unmodified}.

It is noteworthy that even though the behavior is parameter dependent in this example, we can still analyze every possible behavior in our results. 
\end{example}
\clearpage
\section{Supplementary Figures}
\phantom{XXX}\\

\begin{figure}[H]
    \centering
    \includegraphics[width=0.6\linewidth]{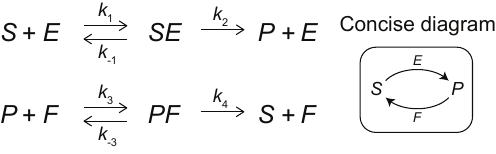}
    \caption{\textbf{A diagram of a phosphorylation-dephosphorylation futile cycle.} This is a well-studied structure that serves as building blocks in cellular signaling pathways \cite{conradi2019multistationarity}. In the box, we put a concise version that demonstrates why it is called a ``futile'' cycle. }
    \label{fig:additional-networks}
\end{figure}
\textbf{SM4. Supplementary Tables}
\begin{table}[H]
    \centering
    \caption{\textbf{Summary of steady-state responses for the unmodified network in Fig.~\ref{fig:additional-networks}.} 
    The left-top $2\times4$ block changes depending on the parameter values. Other possibilities are the following: 
(\textcolor{newred1}{aACR}, $\nearrow\infty$,\textcolor{newred1}{aACR}, $\searrow$0) and 
($\nearrow\infty$, \textcolor{newred1}{aACR}, $\searrow$0, \textcolor{newred1}{aACR}).}
    \begin{tabular}{c|c|c|c|c|c|c|}
         & $S^{\text{SS}}$ & $P^{\text{SS}}$ & $E^{\text{SS}}$ & $F^{\text{SS}}$ & $SE^{\text{SS}}$ & $PF^{\text{SS}}$  \\ \hline
        $S(0)$ & $\nearrow\infty$ & $\nearrow\infty$ & $\searrow$0 & $\searrow$0 & \textcolor{newred1}{aACR} & \textcolor{newred1}{aACR} \\ \hline
        $P(0)$ & $\nearrow\infty$ & $\nearrow\infty$ & $\searrow$0 & $\searrow$0 & \textcolor{newred1}{aACR} & \textcolor{newred1}{aACR} \\ \hline
        $E(0)$ & $\searrow$0 &  \textcolor{newred1}{aACR} & $\nearrow\infty$ & \textcolor{newred1}{aACR} & \textcolor{newred1}{aACR} & \textcolor{newred1}{aACR} \\ \hline
        $F(0)$ &  \textcolor{newred1}{aACR} & $\searrow$0 &  \textcolor{newred1}{aACR} & $\nearrow\infty$ & \textcolor{newred1}{aACR} & \textcolor{newred1}{aACR} \\ \hline
        $SE(0)$ & $\searrow$0 & $\nearrow\infty$ & $\nearrow\infty$ & $\searrow$0 & \textcolor{newred1}{aACR} & \textcolor{newred1}{aACR} \\ \hline
        $PF(0)$ & $\nearrow\infty$ & $\searrow$0 & $\searrow$0 & $\nearrow\infty$  & \textcolor{newred1}{aACR} & \textcolor{newred1}{aACR} \\ \hline
    \end{tabular}
    \label{tab:FutCyc-unmodified}
\end{table}

\vspace{3mm}
\begin{table}[H]
    \centering
    \caption{\textbf{Summary of steady-state concentration responses the network obtained by adding reverse reactions for all of the two irreversible reactions in the unmodified network (Fig.~\ref{fig:additional-networks}).}
    The yellow-colored cells mean their steady-state behaviors are changed from the unmodified network, while transparent cells have the same steady-state behaviors as the unmodified network (see table~\ref{tab:FutCyc-unmodified}.)}
    \begin{tabular}{c|c|c|c|c|c|c|}
         & $S^{\text{SS}}$ & $P^{\text{SS}}$ & $E^{\text{SS}}$ & $F^{\text{SS}}$ & $SE^{\text{SS}}$ & $PF^{\text{SS}}$  \\ \hline
        $S(0)$ & $\nearrow\infty$ & $\nearrow\infty$ & $\searrow$0 & $\searrow$0 & \textcolor{newred1}{aACR} & \textcolor{newred1}{aACR} \\ \hline
        $P(0)$ & $\nearrow\infty$ & $\nearrow\infty$ & $\searrow$0 & $\searrow$0 & \textcolor{newred1}{aACR} & \textcolor{newred1}{aACR} \\ \hline
        $E(0)$ & $\searrow$0 &  \cellcolor{yellow!30}$\searrow$0 & $\nearrow\infty$ & \textcolor{newred1}{aACR} & \textcolor{newred1}{aACR} & \cellcolor{yellow!30}$\searrow$0 \\ \hline
        $F(0)$ &  \cellcolor{yellow!30}$\searrow$0 & $\searrow$0 &  \textcolor{newred1}{aACR} & $\nearrow\infty$ & \cellcolor{yellow!30} $\searrow$0 & \textcolor{newred1}{aACR} \\ \hline
        $SE(0)$ & \cellcolor{yellow!30} $\nearrow\infty$ & $\nearrow\infty$ & $\nearrow\infty$ & $\searrow$0 & \cellcolor{yellow!30} $\nearrow\infty$ & \textcolor{newred1}{aACR} \\ \hline
        $PF(0)$ & $\nearrow\infty$ & \cellcolor{yellow!30}$\nearrow\infty$ & $\searrow$0 & $\nearrow\infty$  & \textcolor{newred1}{aACR} & \cellcolor{yellow!30}$\nearrow\infty$ \\ \hline
    \end{tabular}
    \label{tab:FutCyc-modified-3}
\end{table}


\bibliographystyle{siamplain}
\bibliography{refs}